\documentclass[11pt]{article}

\usepackage[ruled,linesnumbered]{algorithm2e}
\usepackage{float}
\usepackage{graphicx,epsfig}
\usepackage{amsmath}
\usepackage{amssymb}
\usepackage{mathrsfs}

\usepackage{amsmath,amssymb,amsfonts,amsxtra}
\usepackage{amsthm,bm}
\usepackage{graphicx,psfrag}
\usepackage{booktabs}
\usepackage{subfigure}
\usepackage{multirow}
\usepackage{url}
\usepackage{color}
\usepackage[colorlinks,linktocpage,linkcolor=blue]{hyperref}
\usepackage[margin=1.1in]{geometry}
\usepackage{tikz}
\allowdisplaybreaks

\definecolor{cob}{HTML}{0072BD}
\definecolor{coo}{HTML}{D95319}
\definecolor{cog}{HTML}{77AC30}
\definecolor{cop}{HTML}{7E2F8E}
\definecolor{cor}{HTML}{D11717}
\definecolor{coy}{HTML}{EDB120}
\definecolor{coc}{HTML}{00A9CF}

\newcommand{\coltriup}[2][cob]{%
	\begin{tikzpicture}[scale=0.3]
		\fill[#1] (0, 0) -- (1, 0) -- (0.5, 0.866) -- cycle;
	\end{tikzpicture}
}

\definecolor{darkred}{rgb}{.7,0,0}

\definecolor{green}{rgb}{0,0.7,0}

\newtheoremstyle{thmm}{1.5ex plus 1ex minus .2ex}{1.5ex plus 1ex minus.2ex}{\rmfamily}{}{\bfseries}{}{1em}{} \theoremstyle{thmm}
\newtheorem{theorem}{Theorem}[section]
\newtheorem{lemma}{Lemma}[section]

\newtheorem{remark}{Remark}[section]
\renewenvironment{proof}[1][Proof]{\noindent\textit{#1. }
}{\hfill$\square$}

\def\vertt{|\!|\!|}

\def\sign{\operatorname{sign}}

\allowdisplaybreaks

\allowdisplaybreaks[4]

\title{Spectral Method for 1-D Neutron Transport Equation\thanks{The work is supported  by  the National Key R\&D Program of China (No. 2021YFB0300203) and the National Natural Science Foundation of China (NSFC U2230402, NSFC 12471348 and NSFC 12131005).}}
\author{
	Haonan Zhang\thanks{Beijing Computational Science Research Center, Beijing 100193, China. Email: {\tt zhanghaonan@csrc.ac.cn}.},
	\and
	Huiyuan Li\thanks{Corresponding author.  State Key Laboratory of Computer Science/Laboratory of Parallel Computing,  Institute of Software, Chinese Academy of Sciences, Beijing 100190, China. Email: {\tt huiyuan@iscas.ac.cn}.},
	\and
	Zhimin Zhang\thanks{Department of Mathematics, Wayne State University, MI 48202, USA. Email: {\tt  ag7761@wayne.edu}. }
}

\date{}

\begin{document}
	
	\maketitle
	
	\vspace{-10pt}
	
	\begin{abstract}
		In this paper, we present an efficient fully spectral approximation scheme for exploring the one-dimensional steady-state neutron transport equation. Our methodology integrates the spectral-(Petrov-)Galerkin scheme in the spatial dimension with the Legendre-Gauss collocation scheme in the directional dimension. The directional integral in the original problem is discretized with Legendre-Gauss quadrature. We furnish a rigorous proof of the solvability of this scheme and, to our best knowledge, conduct a comprehensive error analysis for the first time. Notably, the order of convergence is optimal in the directional dimension, while in the spatial dimension, it is suboptimal and, importantly, non-improvable. Finally, we verify the computational efficiency and error characteristics of the scheme through several numerical examples.
		\\
	\end{abstract}
	
	\numberwithin{equation}{section}
	\section{Introduction}
	The neutron transport equation comprehensively characterizes interactions involving the movement, scattering, absorption, and leakage of neutrons. Neutrons undergo scattering with atomic nuclei in the medium, resulting in alterations to their direction and energy. Additionally, absorption by the medium contributes to a decrease in the neutron population. Neutrons may also escape system boundaries, influencing the overall transport dynamics. This equation holds fundamental significance in various domains, including nuclear reactor design, nuclear weapons design, medical radiation therapy, nuclear material detection, radiation protection, geophysical exploration, and nuclear technology research. Therefore, achieving a precise understanding and solution to the neutron transport equation is of paramount importance.
	
	Mathematically, the neutron transport equation takes the form of an integro-differential equation, involving both differential and integral operators. In this paper, we focus on the one-dimensional steady transport equation. In the steady state, the neutron population does not change over time, and it is typically assumed that the scattering of neutrons is isotropic. Based on this assumption, we formulate the steady-state neutron transport equation to describe the behavior of neutrons as follows: For $x \in D$, $\mu \in \Lambda$,
	\begin{equation}\label{steady transport}
		\begin{aligned}
			\mathcal{L}\varphi(x,\mu) := \mu\frac{\partial \varphi(x,\mu)}{\partial x}+\Sigma_{t}(x,\mu) \varphi(x,\mu) - \frac{1}{2}\int_{\Lambda} \Sigma_{s}(x,\nu,\mu) \varphi(x,\nu) \mathrm{~d} \nu = \frac{1}{2}s(x,\mu),
		\end{aligned}
	\end{equation}
	where $D \subset \mathbb{R}$ is an open interval, and we suppose $D=(-1,1)$; $\Lambda = (-1,1)$ is the set of all directions of motion, which is the projection of the three-dimensional unit sphere onto $\mathbb{R}$, $\mu$ and $\nu$, belonging to the set $\Lambda$, represent the projection of the direction of motion of neutrons in three dimensions onto the real number line $\mathbb{R}$; $\varphi(x, \mu)$ is the neutron flux density at position $x$, in direction $\mu$; $\Sigma_t(x,\mu)$ is the total macroscopic cross section at position $x$ and direction $\mu$, $\Sigma_s(x, \nu,\mu)$ is the scattering cross section at position $x$ that shifts the direction of neutron motion from $\nu$ to $\mu$, and $s(x, \mu)$ is the neutron source term at position $x$, in direction $\mu$. For neutrons, typically, the total microscopic cross section, $\Sigma_t$, is primarily comprised of the absorption reaction cross section and scattering reaction cross section $\Sigma_s$. The absorption reaction cross section is itself a sum of cross sections for various types of physical reactions, with the most significant being capture and fission. Therefore, the following assumptions are reasonable:
	there exists a constant $c$ such that both $\Sigma_{t}$ and $\Sigma_{s}$ satisfy the following conditions:	
	\begin{equation}\label{hypothesis 1}
		\Sigma_{t}(x, \mu)-\frac{1}{2}\int_{\Lambda}\Sigma_{s}(x,\nu,\mu)d\nu \geq c > 0, \quad  \forall (x,\mu) \in D \times \Lambda,
	\end{equation}
	\begin{equation}\label{hypothesis 2}
		\Sigma_{t}(x, \nu)-\frac{1}{2}\int_{\Lambda}\Sigma_{s}(x, \nu,\mu)d\mu \geq c > 0, \quad  \forall (x,\nu) \in D \times \Lambda.
	\end{equation}	
	Under the above assumptions, the neutron transport equation possesses a unique solution given the vacuum boundary condition:
	\begin{equation}\label{bc}
		\begin{aligned}
			\varphi(-1,\mu) = g_l(\mu) ,\  \mu > 0 \text{ and } \varphi(1,\mu) = g_{r}(\mu),\ \mu < 0.
		\end{aligned}
	\end{equation}
	
	The numerical solution of the neutron transport equation holds significant importance. Two primary types of numerical methods are employed for solving the neutron transport equation. The first category comprises non-deterministic methods, exemplified by the Monte Carlo method (cf. \cite{lewis1984monte, verdu1994review}). The second category encompasses deterministic methods, including the finite difference method (FDM) (cf. \cite{wang2008simulation, yamamoto2003solution}), finite element method (FEM) (cf. \cite{ abbassi2011adaptive,  asadzadeh1998finite,galliara1979finite,lesaint1974finite,martin1977finite,miller1973application}), discontinuous methods (cf. \cite{fournier2013discontinuous,machorro2007discontinuous,reed1973triangularmesh,yuan2016high}). Most of these methods are known for their relatively low accuracy and computational efficiency. 
	
	It is recognized that high-order methods, such as  spectral methods, not only possess an exponential order of convergence for infinitely smooth problems but also provide  high-order accuracy  for non-smooth or even discontinuous solutions in a suitable negative norm. Consequently, researchers have sought to harness high order methods to tackle the neutron transport equation with the goal of enhancing both accuracy and computational efficiency. 	
	
	Cardona proposed an approach in  \cite{cardona1994solution}, which expands the angular flux in the $\mu$ variable using Chebyshev polynomials. Another technique in spatial space to address high-dimensional problems is to decompose them into a series of one- or two-dimensional problems using orthogonal polynomial expansions, coupled with utilizing the discrete ordinates method in the angular direction. This approach, as demonstrated by \cite{vilhena1999solutions} and \cite{kadem2007solution}, reduces the dimensionality of the problem and has been rigorously proven to converge.
	In recent years, there has been tremendous interest in using spectral element method or spectral method in both spatial and directional variables to solve the transport equation. A Legendre polynomial approximation is employed for directional variables, while the Gauss-Lobatto Chebyshev collocation method for spatial variables is utilized to solve the one-dimensional radiative heat transfer equations in \cite{li2008iterative}. Moreover, efforts are being made in software development to solve transport equations using spectral and high-order finite element methods  \cite{ceedYD,MisunMinSEM}. To our knowledge, little literature can be found on the numerical analysis of the convergence and error estimate of a fully  spectral discretization scheme for the transport equation.
	
	The purpose of this paper is to  propose an efficient spectral approximation scheme and then provide a rigorous  error analysis for  the one-dimensional steady-state neutron transport equation. Our approach integrates the spectral (Petrov-)Galerkin method for spatial discretization with the Legendre-Gauss collocation method for directional considerations. Notably, for the integral terms of the equation, we replace continuous integrals with Legendre-Gauss numerical integrals, resulting in a fully spectral discretization.
	For our theoretical purpose, we first construct a projection and present approximation theory related to the projection error. These results serve as a foundation for rigorously proving the solvability theorem for the proposed scheme. An $L^2$-error estimate is subsequently established for our fully spectral approximation scheme, where the order of convergence is optimal in the directional dimension and suboptimal but sharp in the spatial dimension. The suboptimality in the spatial dimension arises from the intricate boundary effect left in a Galerkin approximation. On the other hand, in this particular problem, the convergence rate in the spatial dimension cannot be improved. Finally, we validate the effectiveness of our approach together with numerical analysis theory  through a series of numerical experiments. 


	The remainder of this paper is structured as follows. In Section \ref{sec2}, we present the key properties of Legendre and Jacobi polynomials, which are crucial for developing the spectral method. Section \ref{3} is devoted to deriving the variational form of the spectral method and showcasing its efficacy in solving one-dimensional transport problem. A comprehensive illustrative example is provided, along with a step-by-step elucidation of the implementation process. In Section \ref{4}, we establish error estimates for the one-dimensional transport scenario, meticulously assessing the accuracy of the spectral method. Section \ref{sec5} is dedicated to a series of numerical experiments, which align with our error analysis and demonstrate the performance and dependability of our method. The paper concludes with a summary and conclusion in Section \ref{sec6}.
	
	\section{Preliminaries}\label{sec2}
	This section provides an overview of some preliminary knowledge.
	
	\subsection{Notation}
	We  introduce some notations that will be used throughout this paper. 
	
	Let $\Omega \subset \mathbb{R}^n$ be a bounded domain and $\omega(x)$ be a weight function over $\Omega$. We denote by $\Vert \cdot \Vert_{\omega,\Omega}$  and  $(\cdot, \cdot)_{\omega,\Omega}$ the norm and  inner product in $L^2_{\omega}(\Omega)$ ,
	\begin{align*}
		\Vert u \Vert_{\omega,\Omega} = (u,u)_{\omega,\Omega}^{1/2}, \qquad (u, v)_{\omega,\Omega} = \int_{\Omega}u(x)v(x)\omega(x) dx.
	\end{align*}
	We shall omit the subscripts  $\Omega$ and/or $\omega$ (if $\omega=1$)  whenever no confusion would arise.
	
	Denote by  $P_{N}(D)$ the collection of polynomials on  $D \subset \mathbb{R}$ of degree $\leq N$.
	
	Let $\Lambda = (-1,1)$. For  anisotropic  functions defined over $D \times \Lambda$,  we simply write   $\vertt u \vertt  = \Vert u \Vert_{D\times \Lambda}$, $ \langle u, v \rangle=(u, v)_{D\times \Lambda} $, i.e., 
	\begin{equation*}
		\vertt u \vertt =  \langle u, u \rangle^{1/2}, \qquad \langle u, v \rangle = \int_{\Lambda} \int_{D} u(x,\mu) v(x,\mu) dx d\mu 
		=  \int_{\Lambda}(u(\mu), v(\mu)) d\mu,
	\end{equation*}
	hereafter, the spatial variable of a function $u$ will be dropped whenever no confusion would arise.
	
	Besides,  the following anisotropic Sobolev spaces  will be frequently used:
	\begin{align*}
		&H^{m}(\Lambda;H^{k}(D)) = \big\{ u:  \Lambda  \mapsto H^k(D)\   \big| \       \|u(\mu)\|_{H^k(D)} \in  H^m(\Lambda) \big\}, \\
		&C(\Lambda;L^{2}(D)) = \big\{ u:  \Lambda  \mapsto L^2(D)\   \big| \      \|u(\mu)\|_{L^2(D)} \in C(\Lambda)  \big\}.
	\end{align*}
	Other symbols such  as $P_{M}(\Lambda;P_{N}(D))$ and $H^{p}(\Lambda_{\nu};H^{q}(\Lambda_{\mu};H^{k}(D)))$ can be understood similarly, where $\Lambda_{\mu} = \Lambda$, $\Lambda_{\nu} = \Lambda$ stand for the directional interval with respect to variables $\mu$ and $\nu$, respectively. 
	
	\subsection{Jacobi polynomials}
	Next, we will present several essential formulas for Jacobi and Legendre functions that have been used to solve the transport problem (cf. \cite{shen2011spectral}).

	The Jacobi polynomials, denoted by $J_{n}^{\alpha,\beta}(x)$, $n=0,1,\dots$, for $\alpha,\beta>-1$, are defined as follows:
	$$J_n^{\alpha, \beta}(x) = \frac{\Gamma(n+\alpha+1)}{n ! \Gamma(n+\alpha+\beta+1)}\sum_{k=0}^{n}\binom{n}{k} \frac{\Gamma(n+k+\alpha+\beta+1)}{\Gamma(k+\alpha+1)}\Big(\frac{x-1}{2}\Big)^{k},$$
	which are mutual orthogonal over $D$ with respect to the Jacobi weight function $\omega^{\alpha,\beta}(x) := (1-x)^{\alpha} (1+x)^{\beta}$, and we define $J_{n}^{\alpha,\beta}(x) = 0 \text{, }\text{when } n < 0$. Interested readers are referred to \cite{shen2011spectral}.
	
	Specifically, let $L_{n}(x)$ be the $n$-th Legendre polynomial, which is an important special case of the Jacobi polynomials,
	\begin{align*}
		L_{n}(x) = J_{n}^{0,0}(x), \quad n \ge 0,\quad x \in D = (-1,1).
	\end{align*}
	Legendre polynomials are orthogonal with respect to the weight function $\omega^{0,0}(x) = 1$ over $D$,
	\begin{equation*}
		\int_{-1}^{1} L_n(x)L_m(x) dx = \gamma_n \delta_{m,n}, \quad \gamma_n = \| L_n \|^2 = \frac{2}{2n+1},
	\end{equation*}
	where $\delta_{m n}$ is the Kronecker delta.
	
	The Legendre polynomials and their first derivatives satisfy the following relation: 
	\begin{equation*}\label{Legendre and derivative}
		(2n+1)L_{n}(x) = L_{n+1}^{\prime}(x)-L_{n-1}^{\prime}(x), \quad  n \ge 1. 
	\end{equation*}
	
	\subsection{Useful lemmas}

		At first, we present a trace theory.
	\begin{lemma}\label{boundary error}
		Assuming $\varphi \in L_{\omega^{0,-1}}^{2}(D)$ and $ \partial_x\varphi \in L_{\omega^{1,0}}^{2}(D)$, we have following relation:
		\begin{equation}\label{right}
			\vert \varphi(1) \vert^{2} \lesssim \Vert \varphi \Vert_{  \omega^{0,-1}}\,  \| \partial_x\varphi \big\|_{ \omega^{1,0} };
		\end{equation}
		Similarly, when $\varphi \in L_{\omega^{-1,0}}^{2}(D)$  and $ \partial_x\varphi \in L_{\omega^{0,1}}^{2}(D)$, we have
		\begin{equation}\label{left}
			| \varphi(-1) |^{2} \lesssim \Vert \varphi \Vert_{  \omega^{-1,0} }\,  \| \partial_x\varphi \big\|_{ \omega^{0,1} }.
		\end{equation}
	Hereafter, the notation $A \lesssim B$ means that there exists a generic positive constant $C$, independent of $N$ and any function, such that $A\leq CB$.
	\end{lemma}
	\begin{proof}
		Denote $\Phi_{k}(x) := (1+x)J_{k}^{0,1}(x) = L_{k+1}(x)+L_k(x)$, we can express any $\varphi \in L_{\omega^{0,-1}}^{2}(D)$ as the following expansion:
		$$\varphi(x) = \sum_{k=0}^{\infty} \hat{\varphi}_{k}\Phi_{k}(x).$$
		From \cite[(3.94), (3.106)]{shen2011spectral}, we know
		\begin{equation}\label{value 1}
			|\varphi(1)|^{2} = 4\big(\sum_{k=0}^{\infty}\hat{\varphi}_{k}J_{k}^{0,1}(1)\big)^{2}  = 4\big(\sum_{k=0}^{\infty}\hat{\varphi}_{k}\big)^{2}.
		\end{equation}
		By the orthogonality of $J_{n}^{\alpha,\beta}$ \cite[(3.109)]{shen2011spectral}, it follows that
		\begin{equation}\label{norm L2}
			\begin{aligned}
				\Vert \varphi \Vert_{ \omega^{0,-1} }^{2} = \int_{-1}^{1}\sum_{k=0}^{\infty} \hat{\varphi}_{k}J_{k}^{0,1}(x) \sum_{l=0}^{\infty} \hat{\varphi}_{l}J_{l}^{0,1}(x)(1+x) dx = \sum_{k=0}^{\infty}\hat{\varphi}_{k}^{2}\frac{2}{k+1}.
			\end{aligned}
		\end{equation}
		Owing to the fact that $\Phi_{k}'(x) = (k+1)J_{k}^{1,0}(x)$, it is easy to verify that
		\begin{equation}\label{norm H1}
			\begin{aligned}
				\Vert \partial_x \varphi \Vert_{ \omega^{1,0} }^{2} &= \big(\sum_{k=0}^{\infty}\hat{\varphi}_{k} \Phi_{k}',\sum_{l=0}^{\infty}\hat{\varphi}_{l} \Phi_{l}'\big)_{\omega^{1,0}} = \big(\sum_{k=0}^{\infty}\hat{\varphi}_{k} (k+1)J_{k}^{1,0},\sum_{l=0}^{\infty}\hat{\varphi}_{l} (l+1)J_{l}^{1,0}\big)_{\omega^{1,0}} \\
				&= \sum_{k=0}^{\infty}\hat{\varphi}_{k}^{2} (k+1)^{2} \frac{4}{2k+2} = 2\sum_{k=0}^{\infty} (k+1) \hat{\varphi}_{k}^{2}.
			\end{aligned}
		\end{equation}
		From \eqref{value 1}, \eqref{norm L2} and \eqref{norm H1}, and by Cauchy-Schwarz inequality, we find 
		\begin{equation}\nonumber
			\begin{aligned}
				| \varphi(1) |^{2}
				= 4\big(\sum_{k=0}^{\infty}\hat{\varphi}_{k}\big)^{2} \lesssim \big(\sum_{k=0}^{\infty}\hat{\varphi}_{k}^{2}\frac{2}{k+1}\big)^{\frac{1}{2}}\big(2\sum_{k=0}^{\infty} (k+1) \hat{\varphi}_{k}^{2}\big)^{\frac{1}{2}} =   \Vert \varphi \Vert_{  \omega^{0,-1}}\,  \| \partial_x\varphi \big\|_{ \omega^{1,0} },
			\end{aligned}
		\end{equation}
		which gives the proof of \eqref{right}. Similarly, we derive \eqref{left} and complete the proof.		
	\end{proof}

	Next, denote
	\begin{align*}
		{}^0\!P_{N}(D) = \{u \in P_{N}(D)\ \big|
		\ u(1) = 0 \} \text{, } \quad{}_0P_{N}(D) = \{u \in P_{N}(D) \  \big| \  u(-1) = 0 \}.
	\end{align*}
	We consider two special orthogonal projections: ${}_0\pi_{N}: L_{\omega^{0,-1}}^{2}(D)\mapsto {}_0P_{N}(D)$, 
	\begin{equation}\label{projection 1}
		\begin{aligned}
			(u-{}_0\pi_{N}u, v)=0,\quad \forall v \in P_{N-1}(D), 
		\end{aligned}
	\end{equation}
	and $ {}^0\pi_{N}: L_{\omega^{-1,0}}^{2}(D) \mapsto {}^0\!P_{N}(D)$, 
	\begin{equation}\label{projection 2}
		\begin{aligned}
			(u-{}^0\pi_{N}u,v)=0,\quad \forall v \in P_{N-1}(D).
		\end{aligned}
	\end{equation}
	These projections are then extended for $u \in H^1(D)$ as follows:
	\begin{align}\label{ext orth1}
			&{}_0\pi_N u  = {}_0\pi_N (u - u(-1) ) + u(-1), \\
			\label{ext orth2}
			&{}^0\pi_N u  = {}^0\pi_N (u - u(1) ) + u(1).
	\end{align} 
	For any $u \in H^{1}(D)$ and $v \in P_{N}(D)$, we have following relation,
	\begin{align}
		\label{propIN}
		\begin{split}
			&(\partial_x(u-{}_0\pi_{N}u),  v) =  (u- {}_0\pi_{N}u)(1)  v(1),
			\\
			&(\partial_x(u-{}^0\pi_{N}u),  v) =  ( {}^0\pi_{N}u - u)(-1)  v(-1).
		\end{split}
	\end{align}
	From \cite{shen2011spectral}, we deduce that ${}_0\pi_{N} = \pi_{N}^{0,-1}$ and $ {}^0\pi_{N} = \pi_{N}^{-1,0}$, and derive the following estimations for the truncation errors ${}_0\pi_{N}u - u$ and ${}^0\pi_{N}u - u$. 
	\begin{lemma}\label{error}
		For any $u \in 
		H^m(D)$, we have that for $0 \leq p \leq m \leq N+1$, and $m \geq 1$,
		\begin{equation}\label{Lemma 1}
			\begin{aligned}
				\| \partial_x^p({}_0\pi_N u-u) \|_{\omega^{p, p-1}} \lesssim N^{(p-m)}\left\|\partial_x^m u\right\|_{\omega^{m, m-1}} \lesssim N^{(p-m)}\left\|\partial_x^m u\right\|, \\
				\| \partial_x^p({}^0\pi_N u-u) \|_{\omega^{p-1, p}} \lesssim N^{(p-m)}\left\|\partial_x^m u\right\|_{\omega^{m-1, m}} \lesssim N^{(p-m)}\left\|\partial_x^m u\right\|.
			\end{aligned}
		\end{equation}
	\end{lemma}
	
	Further, denote by $\mu_i,\,  0\le i\le M$  the  zeros of  the  Legendre polynomial  $L_{M+1}(\mu)$,   arranged in the ascending order such that  $-1<\mu_0 <\mu_1 < \dots <\mu_M<1$.  For any $u\in C(\Lambda)$,  we define the Gauss-Legendre  interpolation polynomial $I_{M}u \in P_{M}(\Lambda)$ of $u$ such that
	\begin{align*}
	  (I_{M}u)(\mu_i) =  u(\mu_i), \quad i=0,1,\dots,M.
	 \end{align*}	  
	 Let $\omega_{i}= \frac{2}{(1-\mu_{i}^2) [\partial_{\mu}L_{M+1}(\mu_{i})]^2} $ be  the corresponding  Christoffel numbers. We define the 
	  discrete inner product on $\Lambda$,  
	$$ (u,\phi)_{\Lambda, M}  = \sum_{i=0}^{M}u(\mu_i)\phi(\mu_i)\omega_{i}, \qquad u ,\phi\in C(\Lambda). $$ 	   
	\begin{lemma}[\cite{canuto2007spectral}]\label{difference}
		For any $u \in C(\Lambda)$ and $ \phi \in P_{M}(\Lambda)$, 
		\begin{equation}
			\left|(u, \phi)_{\Lambda}-(u, \phi)_{\Lambda, M} \right| \leq\left\|u-I_M u\right\|_{L^2(\Lambda)}\|\phi\|_{L^2(\Lambda)}.
		\end{equation}
		Further suppose  $u \in H^m(\Lambda)$ with $m \geq 1$, one has
		\begin{equation}\label{interperror}
			\| u-I_M u \|_{L^2(\Lambda)} \lesssim M^{-m} | u |_{H^m(\Lambda)}.
		\end{equation}
	\end{lemma}
	
	\section{Spectral approximation scheme and implementation}\label{3}
	In this section, we will present the variational form and algorithm for the one-dimensional neutron transport equation that correspond to the spectral method.
	
	\subsection{Approximation scheme}
	To approximate the angular direction, we will employ the discrete ordinates method as outlined in \cite{lewis1984computational}. This method entails approximating \eqref{steady transport} for a finite number of distinct angles  and subsequently applying a compatible quadrature approximation to the integral term.
	
	Firstly, the transport equation is satisfied  at the  $M+1$ Legendre-Gauss collocation points in the directional space with the Legendre-Gauss quadrature for the directional integral:  for any $  \mu \in \{ \mu_{i}, i =0, 1,\dots,M \}$,
	$$\mathcal{L}_{M} \varphi (x,\mu) := \mu \frac{\partial \varphi (x,\mu)}{\partial x}+\Sigma_{t}(x,\mu)\varphi (x,\mu)-\frac{1}{2}\sum_{i = 0}^{M}\omega_{i} \Sigma_{s}(x,\nu_{i},\mu)\varphi(x,\nu_{i}) = \frac12s(x,\mu),$$
	where $\nu_m = \mu_{m}, m =0, 1,\dots,M$.
	The  Legendre-Gauss nodes  are  symmetric around the zero, which reflects  the typical equal significance of contributions from right and left particle flows.

	Next,  define the approximation function spaces as follows:
	\begin{equation}
		\begin{aligned}
			& W_{N}^{M} = P_{M}(\Lambda;P_{N}(D)),
			\\
			& P_N^{m}(D) =  {}_0\!P_{N}(D) \text{ if } \mu_m>0 \text{ and }   P_N^{m}(D) =  {}^0\!P_{N}(D) \text{ if } \mu_m<0.
		\end{aligned}
	\end{equation}
	Then fully spectral approximation scheme with boundary conditions being  essentially treated  for the one-dimensional transport equation \eqref{steady transport} reads: to find   $\varphi_{N}^{M} \in W_{N}^{M}$ such that	
	\begin{subequations}\label{Galerkin 1D}
	\begin{align}
			&\big([\mathcal{L}_{M} \varphi_{N}^{M}](\mu_{m}), v_{N}(\mu_{m})\big) = \frac{1}{2}\big(s(\mu_{m}),v_{N} (\mu_{m})\big), \quad v_{N} \in P^m_{N},\ \  m = 0, 1, \cdots, M, 
			\label{Galerkin 1Da} \\
			&  \varphi_{N}^{M}(-1,\mu_{m}) = g_l(\mu_{m}), \quad m = \lceil (M+1)/2 \rceil, \cdots, M,   \label{Galerkin 1Db}  \\
			&  \varphi_{N}^{M}(1,\mu_{m}) = g_r(\mu_{m}), \quad m = 0, \cdots, \lfloor (M-1)/2 \rfloor. \label{Galerkin 1Dc} 
        \end{align}	
	\end{subequations}
	
	Let us introduce the following semi-discrete inner product and norm
	\begin{equation*}
		\langle u, v \rangle_{M} = \sum_{m=0}^{M}\omega_{m}(u(\mu_{m}), v(\mu_{m})), \quad \vertt u \vertt_{M} = \langle u, u \rangle_{M}^{1/2}.
	\end{equation*}
	It is then worthy to point out that  \eqref{Galerkin 1Da} is equivalent to 
	\begin{align}
	\label{Galerkin 1D1}
				&\big\langle  \mathcal{L}_{M} \varphi_{N}^{M}, v_{N}^M\big \rangle_M  = \frac{1}{2}\big \langle  I_{M}^{\mu}s,v_{N}^M  \big \rangle_M , \quad v_{N}^M \in \mathring{W}^M_{N},
	\end{align}
	where $I_{M}^{\mu}$ is the $(M+1)$-point Legendre-Gauss interpolation operator with respect to the directional variable $\mu$, and
	\begin{align*}
		\mathring W_{N}^{M} = \big\{ u \in W_N^M \  :  \ u(\mu_m) \in  P_N^{m}(D),  \, m=0,1,\dots,M \big\}. 
	\end{align*}

Before the conclusion of the current subsection, we  establish a theory on the well-posedness and uniqueness of solution for our approximation scheme.
	\begin{theorem}\label{thm 3.1}
		Let  $\Sigma_{t} \in C(\Lambda;L^{\infty}(D))$, $\Sigma_{s}\in H^{1}(\Lambda_{\nu}; H^{1}(\Lambda_{\mu}; L^{\infty}(D)))$. Under the assumptions \eqref{hypothesis 1} and \eqref{hypothesis 2},  the discrete problem \eqref{Galerkin 1D}  has a unique solution when $M$ is sufficiently large. 
		Specifically, for  \eqref{Galerkin 1D}  with the vacuum boundary conditions,   i.e.,  $\varphi_{N}^{M}\in \mathring{W}_{N}^{ M }$, it holds that
		\begin{equation}\label{ineq}
			\begin{aligned}
				\langle \mathcal{L}_{M}\varphi_{N}^{M}, \varphi_{N}^{M}\rangle_{M} \geq \frac{c}{2} \vertt \varphi_{N}^{M}\vertt^{2}+\frac{1}{2}\sum_{m=0}^{M}\omega_{m}| \mu_{m} |\big(\varphi_{N}^{M}(\sign(\mu_{m}),\mu_{m}) \big)^2.
			\end{aligned}
		\end{equation}
	\end{theorem}  
	\begin{proof}Owing to the linearity of the discrete problem   \eqref{Galerkin 1D}, it suffices to prove \eqref{ineq} for $\varphi_{N}^{M}\in \mathring{W}_{N}^{ M }$.
	
		According to Hille-Yosida theory \cite{dautray1999mathematical}, it is sufficient to prove that for any $ \psi \in \mathring{W}_{N}^{ M } $,
		$$\langle \mathcal{L}_{M}\psi, \psi\rangle_{M} \geq \frac{c}{2} \| \psi \|_{L^{2}(\Lambda;L^{2}(D))}^{2}+\frac{1}{2}\sum_{m=0}^{M}\omega_{m}| \mu_{m} |\big(\psi(\sign(\mu_{m}),\mu_{m}) \big)^2.$$
		Indeed, for any $ \psi \in \mathring{W}_{N}^{ M }$, we have
		\begin{equation*}
			\begin{aligned}
				\langle \mathcal{L}_{M}\psi, \psi\rangle_{M} &= \langle \mu \partial_{x}\psi, \psi \rangle_{M} + \langle \Sigma_{t}\psi, \psi \rangle_{M} \\
				&- \frac{1}{2}\sum_{m=0}^{M}\omega_{m}\sum_{i = 0}^{M}\omega_{i} \big(\Sigma_{s}(\nu_{i},\mu_{m})\psi(\nu_{i}) , \psi(\mu_{m}) \big).
			\end{aligned}
		\end{equation*} 
		Integration by parts and according to the vacuum boundary condition,
		\begin{equation}\label{mu}
			\begin{aligned}
				\langle \mu\partial_{x}\psi, \psi \rangle_{M} = \sum_{m=0}^{M}\omega_{m}\frac{| \mu_{m} |}{2}\big(\psi(\sign(\mu_{m}),\mu_{m}) \big)^2 \geq 0.
			\end{aligned}
		\end{equation}
		Because of $\Sigma_{s} \in H^{1}(\Lambda_{\nu};H^{1}(\Lambda_{\mu};L^{\infty}(D)))$, we know that for almost everywhere $x \in D$,
		\begin{equation*}
			\begin{aligned}
				\bigg| \int_{\Lambda} \Sigma_{s}&(x,\nu,\mu) d\nu - \sum_{i=0}^{M}\omega_{i} \Sigma_{s}(x,\nu_{i},\mu) \bigg| 
				= \bigg| \int_{\Lambda} (I-I_{M}^{\nu})\Sigma_{s}(x,\nu,\mu) d\nu \bigg| \\
				\leq & 2\bigg( \int_{\Lambda} \big((I-I_{M}^{\nu})\Sigma_{s}(x,\nu,\mu)\big)^{2} d\nu \bigg)^{1/2} 
				\lesssim M^{-1} \bigg( \int_{\Lambda} \big(\partial_{\nu}\Sigma_{s}(x,\nu,\mu)\big)^{2} d\nu \bigg)^{1/2},
			\end{aligned}
		\end{equation*}
		where $I_{M}^{\nu}$ is the $(M+1)$-point Legendre-Gauss interpolation operator with respect to variable $\nu$. 
		
		According to \eqref{hypothesis 1} and \eqref{hypothesis 2}, and owing to $\Sigma_{t}(\cdot,\mu) \in L^{\infty}(D)$, we can observe that: when $M$ is sufficiently large, for $\mu \in \{\mu_{m}, m = 0, 1,\dots, M\}$ and almost everywhere $x \in D$,
		\begin{equation}\label{assume1}
			\begin{aligned}
				\Sigma_{t}&(x,\mu)-\frac{1}{2}\sum_{i=0}^{M}\omega_{i} \Sigma_{s}(x,\nu_{i},\mu) \\
				&= \Sigma_{t}(x,\mu)-\frac{1}{2} \int_{\Lambda} \Sigma_{s}(x,\nu,\mu) d\nu +  \frac{1}{2}\bigg( \int_{\Lambda} \Sigma_{s}(x,\nu,\mu) d\nu - \sum_{i=0}^{M}\omega_{i} \Sigma_{s}(x,\nu_{i},\mu) \bigg) \\
				& \geq \frac{c}{2}.
			\end{aligned}
		\end{equation}
		Similarly, it is easy to see that for $\nu \in \{ \nu_{m}, m = 0, \cdots, M\}$
		\begin{equation}\label{assume2}
			\Sigma_{t}(x,\nu)-\frac{1}{2}\sum_{i=0}^{M}\omega_{i} \Sigma_{s}(x,\nu,\mu_{i}) \geq \frac{c}{2}, \quad \text{a.e. } x \in D.
		\end{equation}
		On the other hand, by using the Cauchy-Schwarz inequality and according to \eqref{assume1} and \eqref{assume2}, we notice that 
		\begin{equation}\label{sigma}
			\begin{aligned}
				&\big|\sum_{m,i=0}^{M}\omega_{m}\omega_{i} \Sigma_{s}(x,\nu_{i},\mu_{m})\psi(x,\nu_{i})\psi(x,\mu_{m}) \big| \\
				&\leq \Big(\sum_{m, i=0}^{M}\omega_{m}\omega_{i}\Sigma_{s}(x,\nu_{i},\mu_{m})
				(\psi(x,\nu_{i}))^2 \Big)^{1/2}  \Big( \sum_{m, i=0}^{M}\omega_{m}\omega_{i}\Sigma_{s}(x,\nu_{i},\mu_{m})(\psi(x,\mu_{m}))^2 \Big)^{1/2} \\
				& \leq 2\sum_{m=0}^{M}\omega_{m}\Sigma_{t}(x,\mu_{m})(\psi(x,\mu_{m}))^{2}-c\sum_{m=0}^{M}\omega_{m}(\psi(x,\mu_{m}))^{2}, \qquad \text{a.e. } x \in D.
			\end{aligned}
		\end{equation}
		Integrating equation \eqref{sigma} over $D$ with respect to $x$, and noting $\vertt \psi \vertt_{M} = \vertt \psi \vertt$,
		\begin{equation*}
			\begin{aligned}
				\langle \mathcal{L}&_{M}\psi,\psi\rangle_{M} \geq \frac{1}{2}\sum_{m=0}^{M}\omega_{m}| \mu_{m} |\big(\psi(\sign(\mu_{m}),\mu_{m})\big)^{2} + \sum_{m=0}^{M}\omega_{m}(\Sigma_{t}(\mu_{m})\psi(\mu_{m}),\psi(\mu_{m})) \\  
				&\quad - \frac{1}{2}\big( 2\sum_{m=0}^{M}\omega_{m}(\Sigma_{t}(\mu_{m})\psi(\mu_{m}),\psi(\mu_{m}))- c \vertt \psi \vertt_{M}^2 \big) \\ 
				& \geq \frac{c}{2}\vertt \psi \vertt^{2} + \frac{1}{2}\sum_{m=0}^{M}\omega_{m}| \mu_{m} |\big(\psi(\sign(\mu_{m}),\mu_{m}) \big)^2.
			\end{aligned}
		\end{equation*} 
		Specifically, taking $\psi = \varphi_{N}^{M}$, we derive \eqref{ineq}. The proof is completed. 
	\end{proof}
	\subsection{Implementation}

	Define the basis functions $\{\Phi_{n}(x)\}_{n=0}^{N}$ on the interval $D = (-1,1)$ as follows:
	\begin{equation}\label{expand2}
		\begin{aligned}
			\begin{split}
				\Phi_{n}(x) = \Phi_{n}^{N}(x):= \left\{ 
				\begin{array}{ll}
					\frac{1-x}{2},    & n = 0, \\[0.5em]	
					\frac{1}{\sqrt{4n+2}}\big(L_{n-1}(x)-L_{n+1}(x)\big),  & 0 < n < N, \\[0.8em]	
					\frac{1+x}{2},    & n = N.
				\end{array}
				\right.
			\end{split}
		\end{aligned}
	\end{equation}
	We expand the approximation solution $\varphi_N^M \in W_N^M$ to the unknown function $\varphi(x,\mu)$ in terms of the basis function $\Phi_{n}(x), n = 0,1,\cdots, N$,
	\begin{equation}\label{expand}
		\varphi_{N}^{M}(x,\mu) := \sum_{n=0}^{N}\hat{\varphi}_{n}(\mu) \Phi_{n}(x) \in W_{N}^{M},
	\end{equation}
	with the inflow  boundary condition to be satisfied explicitly in the final algebraic system. 
	 
	Based on \eqref{Galerkin 1D}, we provide the structure of the coefficient matrix for the neutron transport equation:
	\begin{equation*}
		\begin{aligned}
			&B = diag( \mu_0\hat{B}, \mu_1\hat{B}, \cdots, \mu_M\hat{B} ), && \hat{B} := \big[\big(\partial_{x}\Phi_{j}, \Phi_{i} \big) \big]_{0 \leq i,j \leq N}, \\
			&T = diag(T_{0}, T_{1}, \cdots, T_{M} ), && T_{m} := \big[ \big(\Sigma_{t}(\mu_{m})\Phi_{j}, \Phi_{i} \big) \big]_{0 \leq i,j \leq N}, \\
			&S = [S_{m,n}]_{0 \leq m,n \leq M}, && S_{m,n} := \big[ \frac{1}{2} \omega_{n}\big(\Sigma_{s}(\nu_{n}, \mu_{m})\Phi_{j}, \Phi_{i} \big) \big]_{0 \leq i,j \leq N}, 
		\end{aligned}
	\end{equation*}
	where $T$ and $S$ are symmetric matrices. 
	
	Denote 
	\begin{equation*}
		\begin{aligned}
			&\mathbf{u} = (\hat{\varphi}_{0,0}, \hat{\varphi}_{1,0}, \cdots, \hat{\varphi}_{N,0}, \hat{\varphi}_{0,1}, \cdots, \hat{\varphi}_{N,M})^{T}, \quad \hat{\varphi}_{j,m} = \hat{\varphi}_{j}(\mu_{m}), \quad 0 \leq j \leq N, 0 \leq m \leq M, \\
			&\mathbf{f} = (f_{0,0}, f_{1,0}, \cdots, f_{N,0}, f_{0,1}, \cdots, f_{N,M})^{T}, \quad f_{i,m} = \frac{1}{2}\big(s(\mu_{m}),\Phi_{i} \big), \quad 0 \leq i \leq N, 0 \leq m \leq M.
		\end{aligned}
	\end{equation*}
	Denote $DoF=(N+1)(M+1)-2\lfloor (M+1)/2\rfloor$.
	Let $\mathcal{Q} \in \mathbb{R}^{DoF \times(N+1)(M+1)}$ be the mask matrix  is obtained by 
	removing the $ 2\lfloor (M+1)/2\rfloor$  rows associated with the inflow boundary conditions from 
	  the identity matrix of size $(N+1)(M+1)\times(N+1)(M+1)$.  Meanwhile, let $\mathcal{R} \in \mathbb{R}^{2\lfloor (M+1)/2\rfloor \times(N+1)(M+1)}$ be the complementary matrix of  $\mathcal{Q}$,  which retains only the rows  of the identity matrix corresponding to the inflow boundary condition.
	
	Then we get the matrix form of \eqref{Galerkin 1D}: 
	\begin{align*}
		&\big[\mathcal{Q} \bf A\mathcal{Q}^T\big]( \mathcal{Q}\bf{u}) = \mathcal{Q}\big [\bf{f} - A\mathcal{R}^T(\mathcal{R}u)\big], \quad A = B+T-S,
		\\
		& \mathcal{R} \mathbf{u} = ( g_r(\mu_{0}), \cdots, g_r(\mu_{\lfloor (M-1)/2 \rfloor}), g_l(\mu_{\lceil (M+1)/2 \rceil}), \cdots, g_l(\mu_{M}) )^T.
	\end{align*}
	 For example, when we take $N = 5$, $M = 3$, Figure \ref{matbc} gives the global indexing of  the unknowns, along with the corresponding mask matrices $ \mathcal{Q}$ and $ \mathcal{R}$.
	\begin{figure}[H]
		\begin{minipage}{0.5\textwidth}
		\centering
		\includegraphics[width=0.7\textwidth]{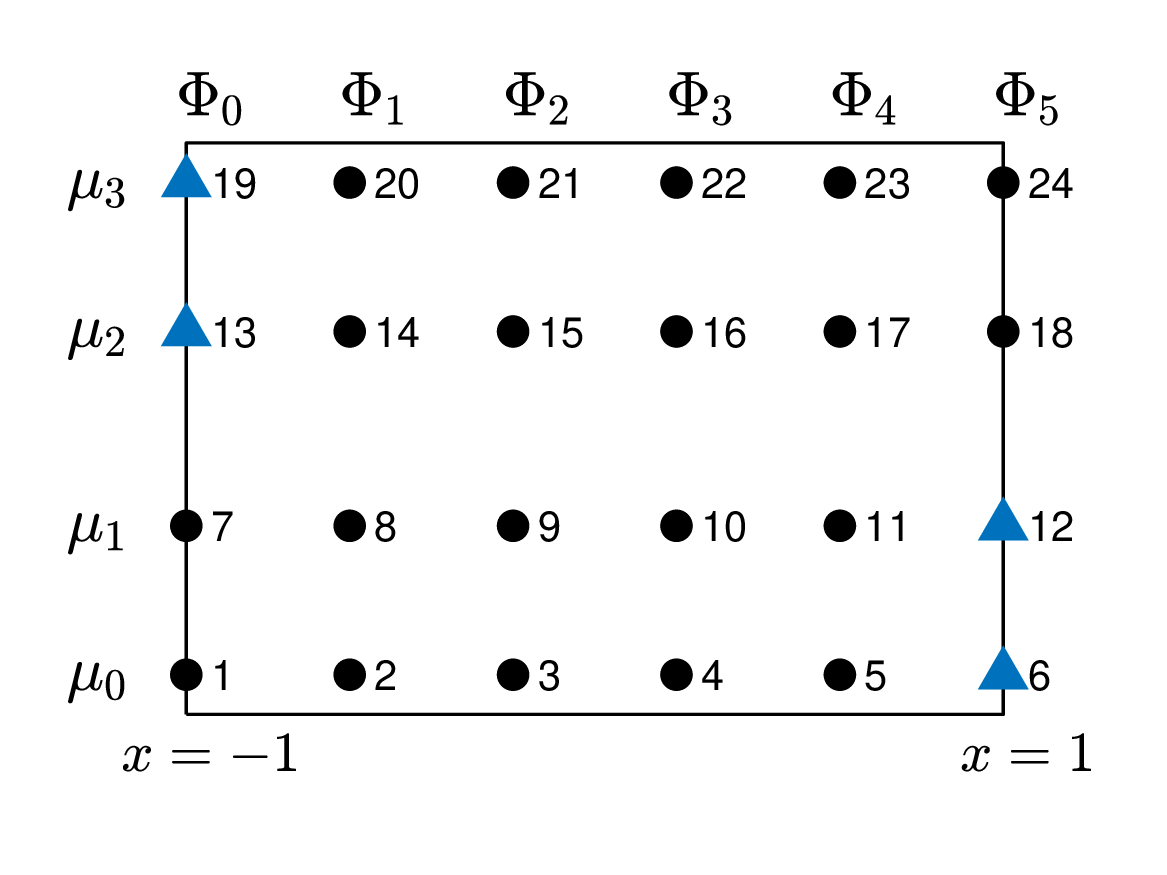}
		\end{minipage}\begin{minipage}{0.5\textwidth}
		\begin{align*}
		   &idx_{io} = [1:5,\ 7:11,\ 14:18,\ 20:24], \\
		    &idx_{bc} = [6, 12, 13, 19],
		   \\
		   &\mathcal{Q} = I_{24}(idx_{io},:), 
		   \\
		    &\mathcal{R} = I_{24}(idx_{bc},:).
		\end{align*}
		\end{minipage}
		\caption{Global indexing of  the unknowns and the  mask matrices. \coltriup[cob]: the inflow boundary condition; $\bullet$: interior and the outflow boundary.}
		\label{matbc}
	\end{figure}

	\section{Error estimate}\label{4}
	In this section, we will present the error estimate for the one-dimensional steady-state neutron transport problem.
	
	Let us first consider the orthogonal projection. When $\mu > 0$, $ \pi_{N}^\mu := {}_0\pi_{N}$ defined in \eqref{projection 1} and \eqref{ext orth1}, while when $\mu < 0$, $\pi_{N}^\mu := {}^0\pi_{N}$ defined in \eqref{projection 2} and \eqref{ext orth2}.  When there is no need to distinguish between the positive and negative values of $\mu$, we present $\pi_{N}^\mu$  simply as $\pi_{N}$.
	
	To simplify the formulas, we introduce two operators, $Q$ and $Q_{M}$:
	\begin{align*}
		&[Qf](x,\mu) = \frac12\int_{\Lambda}\Sigma_{s}(x,\nu,\mu)f(x,\nu) d\nu,\\
		&[Q_{M}f](x,\mu) = \frac12 \sum_{i=0}^{M}\omega_{i}\Sigma_{s}(x,\nu_{i},\mu)f(x,\nu_{i}) = \frac12 \int_{\Lambda}I_{M}^{\nu}\Sigma_{s}(x,\nu,\mu)I_{M}^{\nu}f(x,\nu) d\nu.
	\end{align*}
	Denote $\hat{e}_{N}^{M} = \varphi_{N}^{M}-I_{M}^{\mu}\pi_{N}\varphi \in \mathring{W}_N^M$. We derive from  \eqref{steady transport} and \eqref{Galerkin 1D1} that

	\begin{equation}
	\label{1st}
		\begin{split}
			\langle \mathcal{L}_{M}\hat{e}_{N}^{M},v_{N}^{M} \rangle_{M}
			=&\, \big[ \langle \mathcal{L}_M \varphi_N^M, v_{N}^{M} \rangle_{M} - \langle \mathcal{L} \varphi, v_{N}^{M} \rangle\big] +\big[ \langle \mathcal{L} \varphi, v_{N}^{M} \rangle  - \langle \mathcal{L}_{M}I_{M}^{\mu}\pi_{N} \varphi, v_{N}^{M} \rangle_{M}\big] \\
			=&\,\frac12\langle (I_{M}^{\mu} - I)s, v_{N}^{M} \rangle +[ \langle \mu\partial_{x}\varphi, v_{N}^{M} \rangle - \langle \mu\partial_{x}I_{M}^{\mu}\pi_{N}\varphi, v_{N}^{M} \rangle _{M} ] \\
			& +[ \langle \Sigma_{t}\varphi, v_{N}^{M} \rangle - \langle \Sigma_{t}I_{M}^{\mu}\pi_{N}\varphi, v_{N}^{M} \rangle_{M} ]  + [ \langle Q_{M}I_{M}^{\nu}\pi_{N}\varphi, v_{N}^{M} \rangle_{M} - \langle Q\varphi, v_{N}^{M} \rangle ].
		\end{split}
	\end{equation}
	Noting that  $ \langle \mu u, v\rangle_M=\langle \mu u,v \rangle$ for any $u,v\in W_N^M$ , we know from  \eqref{propIN} that
	\begin{equation}
		\begin{aligned}
			\langle \mu&\partial_{x}\varphi, v_{N}^{M} \rangle - \langle \mu\partial_{x}I_{M}^{\mu}\pi_{N}\varphi, v_{N}^{M} \rangle_M
		= \langle \mu \partial_{x}  (I-I_{M}^{\mu})  \varphi, v_{N}^{M} \rangle + \langle \mu\partial_{x}I_{M}^{\mu}(I-\pi_{N})\varphi, v_{N}^{M} \rangle_M
			 \\
			&=  \langle \mu \partial_{x}  (I-I_{M}^{\mu})  \varphi, v_{N}^{M} \rangle +\sum_{i=0}^M   \omega_i | \mu_i| \big[  v_{N}^{M} I_{M}^{\mu} (I - \pi_{N}) \varphi\big] (\sign(\mu_i),\mu_i)  ,
		\end{aligned}
	\end{equation}
	where have used that $[( \pi_N - I )\varphi] (-1,\mu_{m}) = 0$ when $\mu_{m} > 0$, and  $[( \pi_N - I )\varphi](1,\mu_{m}) = 0$ when $\mu_{m} < 0$.


	For terms that contain $\Sigma_{t}$ in \eqref{1st}, we have
	\begin{equation}
		\begin{aligned}
			\langle \Sigma&_{t}\varphi, v_{N}^{M} \rangle - \langle \Sigma_{t}I_{M}^{\mu}\pi_{N}\varphi, v_{N}^{M} \rangle_{M}
			= \langle (I-I_{M}^{\mu})(\Sigma_{t}\varphi), v_{N}^{M} \rangle + \langle \Sigma_{t}(\varphi-I_{M}^{\mu}\pi_{N}\varphi), v_{N}^{M} \rangle_{M}.
		\end{aligned}
	\end{equation}
	Similarly, 
	\begin{equation}
	\label{Qequality}
		\begin{aligned}
			\langle Q_{M}&I_{M}^{\nu}\pi_{N}\varphi, v_{N}^{M} \rangle_{M} - \langle Q\varphi, v_{N}^{M} \rangle
			= \langle Q_{M}(I_{M}^{\nu}\pi_{N}-I)\varphi, v_{N}^{M} \rangle_{M} + \langle I_{M}^{\mu}(Q_{M}\varphi) - Q\varphi, v_{N}^{M} \rangle.
		\end{aligned}
	\end{equation}
	Taking $v_{N}^{M} = \hat{e}_{N}^{M}$ in \eqref{1st}-\eqref{Qequality} and employing \eqref{ineq} along with the Cauchy-Schwarz inequality, we derive the error equation when $M$ is sufficiently large,
	\begin{equation}
	\label{L2_Bnd_Bounds}
		\begin{aligned}
			\frac{c}{2} &\vertt \hat{e}_{N}^{M} \vertt^{2} + \frac12\sum_{i=0}^{M} \omega_{i} |\mu_{i}| \, |\hat{e}_{N}^{M}(\sign(\mu_{i}),\mu_{i})|^{2}  \leq 	\langle \mathcal{L}_{M}\hat{e}_{N}^{M},\hat{e}_{N}^{M} \rangle_{M} 
			\\
			& \leq \frac12 \vertt (I_{M}^{\mu}-I)s \vertt \vertt \hat{e}_{N}^{M} \vertt + \vertt \partial_{x}(I-I_{M}^{\mu})\varphi \vertt \vertt \hat{e}_{N}^{M} \vertt \\
			&\quad + \bigg(\sum_{i=0}^{M}\omega_{i} |\mu_{i}| \, | [I_{M}^{\mu}(I - \pi_{N})\varphi] (\sign(\mu_{i}),\mu_{i})|^{2} \bigg)^{\frac12} \bigg(\sum_{i = 0}^{M} \omega_{i} |\mu_{i}| (\hat{e}_{N}^{M}(\sign(\mu_{i}),\mu_{i}))^{2}\bigg)^{\frac12}	\\
			&\quad + \vertt (I-I_{M}^{\mu})(\Sigma_{t}\varphi) \vertt \vertt \hat{e}_{N}^{M} \vertt + \vertt \Sigma_{t}(I-I_{M}^{\mu}\pi_{N})\varphi \vertt_{M} \vertt \hat{e}_{N}^{M} \vertt_{M} \\
			& \quad + \vertt Q_{M}(I_{M}^{\nu}\pi_N - I)\varphi \vertt_{M} \vertt \hat{e}_{N}^{M} \vertt_{M} + \vertt I_{M}^{\mu}(Q_{M}\varphi) - Q\varphi \vertt \vertt \hat{e}_{N}^{M} \vertt,
\\
			& \leq \frac{1}{16\varepsilon} \vertt (I_{M}^{\mu}-I)s \vertt^2 + \varepsilon \vertt \hat{e}_{N}^{M} \vertt^2 + \frac{1}{4\varepsilon} \vertt \partial_{x}(I-I_{M}^{\mu})\varphi \vertt^2 + \varepsilon \vertt \hat{e}_{N}^{M} \vertt^2 \\
			&\quad + \sum_{i=0}^{M}\omega_{i} |\mu_{i}| \, | I_{M}^{\mu}(I - \pi_{N})\varphi(\sign(\mu_{i}),\mu_{i})|^{2}   + \frac14 \sum_{i = 0}^{M} \omega_{i} |\mu_{i}| (\hat{e}_{N}^{M}(\sign(\mu_{i}),\mu_{i}))^{2} \\
			&\quad +\frac{1}{4\varepsilon} \vertt (I-I_{M}^{\mu})(\Sigma_{t}\varphi) \vertt^2 + \varepsilon \vertt \hat{e}_{N}^{M} \vertt^2 + \frac{1}{4\varepsilon} \vertt \Sigma_{t}(I-I_{M}^{\mu}\pi_{N})\varphi \vertt_{M}^2 + \varepsilon \vertt \hat{e}_{N}^{M} \vertt^2 \\
			& \quad + \frac{1}{4\varepsilon} \vertt Q_{M}(I_{M}^{\nu}\pi_N - I)\varphi \vertt_{M}^2 + \varepsilon \vertt \hat{e}_{N}^{M} \vertt^2 + \frac{1}{4\varepsilon} \vertt I_{M}^{\mu}(Q_{M}\varphi) - Q\varphi \vertt^2 + \varepsilon \vertt \hat{e}_{N}^{M} \vertt^2,	
		\end{aligned}
	\end{equation}
	where we have used  the basic inequality $ab \leq \frac{1}{4\varepsilon} a^2 + \varepsilon b^2$ for any $\varepsilon > 0$, 
	together with  $\vertt \hat{e}_{N}^{M} \vertt_{M} = \vertt \hat{e}_{N}^{M} \vertt$ for the second inequality sign.

	 Taking $\varepsilon = \frac{c}{24}$ in \eqref{L2_Bnd_Bounds}  and noting that
	\begin{align*}
			&\vertt \varphi - \varphi_{N}^{M} \vertt^2  \leq 2 \vertt \varphi - I_{M}^{\mu}\pi_{N}\varphi \vertt^2 + 2\vertt \hat{e}_{N}^{M} \vertt^2,
			\\
				 & \sum_{i = 0}^M \omega_{i} |\mu_{i}| \big( (\varphi - \varphi_{N}^{M})(\sign(\mu_{i}),\mu_{i}) \big)^2
				  \le  2\sum_{i = 0}^M \omega_{i} |\mu_{i}| \big( \hat{e}_{N}^{M} (\sign(\mu_{i}),\mu_{i}) \big)^2
				\\
				&\qquad\qquad   +  2\sum_{i = 0}^M \omega_{i} |\mu_{i}| \big( [ I_M^{\mu} (I - \pi_{N})\varphi ] (\sign(\mu_{i}),\mu_{i}) \big)^2 ,
		\end{align*} 
	we arrive at  the following estimate:
	\begin{equation}\label{2st}
		\begin{aligned}
			\frac{c}{4} \vertt & \varphi - \varphi_{N}^{M} \vertt^{2} + \frac14 \sum_{i=0}^{M} \omega_{i} |\mu_{i}| \, |[\varphi - \varphi_{N}^{M}](\sign(\mu_{i}),\mu_{i})|^{2} 
			\\
				\leq&\,  \frac3{c} \vertt (I_{M}^{\mu}-I)s \vertt^{2}  + \frac{12}c \vertt \partial_{x}(I-I_{M}^{\mu})\varphi \vertt^{2}
			+ \frac{c}2  \vertt  \varphi - I_{M}^{\mu}\pi_{N}\varphi \vertt^2 
			\\
			& + \frac{12}c \vertt (I-I_{M}^{\mu})(\Sigma_{t}\varphi) \vertt^{2} + \frac{12}c \vertt \Sigma_{t}(I-I_{M}^{\mu}\pi_{N})\varphi \vertt_{M}^{2} + \frac{12}c \vertt Q_{M}(I_{M}^{\nu}\pi_N - I)\varphi \vertt_{M}^{2} \\
			& + \frac{12}c \vertt I_{M}^{\mu}(Q_{M}\varphi) - Q\varphi \vertt^{2} + \frac52\sum_{i=0}^{M}\omega_{i} |\mu_{i}| \, | I_{M}^{\mu}(I - \pi_{N})\varphi(\sign(\mu_{i}),\mu_{i})|^{2}.
		\end{aligned}
	\end{equation}
	We divide \eqref{2st} into eight terms. From the \eqref{interperror} and Lemma \ref{error}, it is easy to show that 
	\begin{equation}\label{inequality 1}
		\begin{aligned}
			& \vertt (I_{M}^{\mu} - I)s \vertt \lesssim M^{-q}\| s \|_{H^{q}(\Lambda;L^{2}(D))}, \quad q \geq 1, \\
			& \vertt \partial_{x}(I-I_{M}^{\mu})\varphi \vertt \lesssim M^{-p}\| \varphi \|_{H^{p}(\Lambda;H^{1}(D))}, \quad p \geq 1, \\
			& \vertt  \varphi - I_{M}^{\mu}\pi_{N}\varphi \vertt  \le \vertt  (I - I_{M}^{\mu})\varphi \vertt +  \vertt (I_{M}^{\mu}-I)(I - \pi_{N})\varphi \vertt+ \vertt (I - \pi_{N})\varphi \vertt,
			\\
			& \quad \lesssim M^{-\sigma} \| \varphi \|_{H^{\sigma}(\Lambda;L^2(D))}+ M^{-p}N^{-1} \| \varphi \|_{H^{p}(\Lambda;H^{1}(D))}
			+ N^{-k} \| \varphi \|_{L^{2}(\Lambda;H^{k}(D))}, \quad k \geq 1.
		\end{aligned}
	\end{equation}
	When $\Sigma_{t}, \varphi \in H^{\sigma}(\Lambda;L^2(D))$, $\sigma \geq 1$, we can derive from the \eqref{interperror} that
	\begin{equation}
		\begin{aligned}
			\vertt (I-I_{M}^{\mu})(\Sigma_{t}\varphi) \vertt & \lesssim M^{-\sigma} \| \Sigma_{t}\varphi \|_{H^{\sigma}(\Lambda;L^2(D))} \\
			&\lesssim M^{-\sigma} \| \Sigma_{t} \|_{H^{\sigma}(\Lambda;L^2(D))} \| \varphi \|_{H^{\sigma}(\Lambda;L^2(D))}.
		\end{aligned}
	\end{equation}
	From the assumption $\Sigma_{t} \in C(\Lambda;L^{\infty}(D))$, we know that for almost everywhere $x \in D$ and any $\mu \in \Lambda$, $|\Sigma_{t}(x,\mu)| \leq \|\Sigma_{t} \|_{C(\Lambda;L^{\infty}(D))}$, and due to Lemma \ref{error}, we derive
	\begin{equation}
		\begin{aligned}
			\vertt \Sigma_{t}&(I-I_{M}^{\mu}\pi_{N})\varphi \vertt_{M} \leq \|\Sigma_{t}\|_{C(\Lambda;L^{\infty}(D))} \vertt (I-I_{M}^{\mu}\pi_{N})\varphi \vertt_{M}  \\
			&= \|\Sigma_{t}\|_{C(\Lambda;L^{\infty}(D))} \vertt (I_{M}^{\mu}-I_{M}^{\mu}\pi_{N})\varphi \vertt \\
			&\leq \|\Sigma_{t}\|_{C(\Lambda;L^{\infty}(D))} (\vertt (I-I_{M}^{\mu})(I - \pi_{N})\varphi \vertt + \vertt (I-\pi_{N})\varphi \vertt ) \\
			&\lesssim \|\Sigma_{t}\|_{C(\Lambda;L^{\infty}(D))}(M^{-p}N^{-1} \| \varphi \|_{H^{p}(\Lambda;H^{1}(D))} + N^{-k} \| \varphi \|_{L^{2}(\Lambda;H^{k}(D))} ), \quad p,k \geq 1.
		\end{aligned}
	\end{equation}
	And when $\Sigma_{s} \in H^{1}(\Lambda_{\nu};H^{1}(\Lambda_{\mu};L^{\infty}(D)))$, it is obvious that for almost everywhere $x \in D$ and all $\mu \in \Lambda$, $\nu \in \Lambda$, $| \Sigma_{s}(x,\mu,\nu) | \leq \| \Sigma_{s} \|_{C(\Lambda_{\nu};C(\Lambda_{\mu};L^{\infty}(D)))}$. So when $p,k \geq 1$, we have
	\begin{equation}
		\begin{aligned}
			\vertt &Q_{M}(I_{M}^{\nu}\pi_{N}-I)\varphi \vertt_{M}^{2} = \vertt Q_{M}(I_{M}^{\nu}\pi_{N}-I_{M}^{\nu})\varphi \vertt_{M}^{2} \\
			&\leq \frac12 \sum_{m = 0}^{M} \omega_{m} \int_{D} \big(\sum_{i=0}^{M}\omega_{i} | \Sigma_{s}(x,\nu_{i},\mu_{m}) |\, | ( I_{M}^{\nu}\pi_{N}-I_{M}^{\nu})\varphi(x,\nu_{i}) | \big)^{2}dx \\
			&\lesssim \| \Sigma_{s} \|_{C(\Lambda_{\nu};C(\Lambda_{\mu};L^{\infty}(D)))}^{2} \| I_{M}^{\nu}(\pi_{N}-I)\varphi \|_{L^{2}(\Lambda;L^{2}(D))}^{2} \\
			&\lesssim \| \Sigma_{s} \|_{C(\Lambda_{\nu};C(\Lambda_{\mu};L^{\infty}(D)))}^{2} (\| (I_{M}^{\nu}-I)(\pi_{N}-I)\varphi \|_{L^{2}(\Lambda;L^{2}(D))}^{2} + \| (\pi_{N}-I)\varphi \|_{L^{2}(\Lambda;L^{2}(D))}^{2}) \\
			&\lesssim \| \Sigma_{s} \|_{C(\Lambda_{\nu};C(\Lambda_{\mu};L^{\infty}(D)))}^{2}(M^{-2p}N^{-2} \| \varphi \|_{H^{p}(\Lambda;H^{1}(D))}^{2} + N^{-2k} \| \varphi \|_{L^{2}(\Lambda;H^{k}(D))}^{2} ).
		\end{aligned}
	\end{equation}
	When $\Sigma_{s} \in H^{\sigma}(\Lambda_{\nu};L^2(\Lambda_{\mu};L^{2}(D))) \cap L^{2}(\Lambda_{\nu};H^{\gamma}(\Lambda_{\mu};L^{2}(D))) \cap H^{\sigma'}(\Lambda_{\nu};H^{\gamma'}(\Lambda_{\mu};L^{2}(D)))$, $\varphi \in H^{\max\{\sigma,\sigma'\}}(\Lambda;L^{2}(D))$, and $\sigma, \gamma, \sigma', \gamma' \geq 1$, it can easily be shown that
	\begin{equation}
		\begin{aligned}
			\vertt I_{M}^{\mu}&(Q_{M}\varphi) - Q\varphi \vertt^{2} = \frac12 \int_{\Lambda} \int_{D} \big(\int_{\Lambda} (I_{M}^{\mu}I_{M}^{\nu}-I)(\Sigma_{s}(x,\nu,\mu)\varphi(x,\nu))  d\nu\big)^{2} dx d\mu \\
			& \lesssim \int_{\Lambda} \int_{D} \big(\int_{\Lambda} (I_{M}^{\nu}-I)((I_{M}^{\mu}-I)\Sigma_{s}(x,\nu,\mu)\varphi(x,\nu))  d\nu\big)^{2} dx d\mu \\
			&\quad + \int_{\Lambda} \int_{D} \big(\int_{\Lambda} ((I_{M}^{\mu}-I)\Sigma_{s}(x,\nu,\mu)\varphi(x,\nu))  d\nu\big)^{2} dx d\mu \\
			& \quad + \int_{\Lambda} \int_{D} \big(\int_{\Lambda}  (I_{M}^{\nu}-I) (\Sigma_{s}(x,\nu,\mu)\varphi(x,\nu))  d\nu\big)^{2} dx d\mu \\
			&\lesssim M^{-2(\sigma'+\gamma')} \| \Sigma_{s} \|_{  H^{\sigma'}(\Lambda_{\nu};H^{\gamma'}(\Lambda_{\mu};L^{2}(D)))  }^{2} \| \varphi \|_{H^{\sigma'}(\Lambda;L^{2}(D))}^{2} \\
			& \quad + M^{-2\sigma} \| \Sigma_{s} \|_{  H^{\sigma}(\Lambda_{\nu};L^2(\Lambda_{\mu};L^{2}(D)))  }^{2} \| \varphi \|_{H^{\sigma}(\Lambda;L^{2}(D))}^{2} \\
			&\quad+ M^{-2\gamma} \| \Sigma_{s} \|_{L^{2}(\Lambda_{\nu};H^{\gamma}(\Lambda_{\mu};L^{2}(D)))}^{2} \| \varphi \|_{L^{2}(\Lambda;L^{2}(D))}^{2}, 
		\end{aligned}
	\end{equation}
	where the last inequality is derived by using  the error estimate \eqref{interperror} and the Cauchy-Schwarz inequality.
	%
	
	Finally, from Lemma \ref{error}, it is obvious that for $\mu \in \{ \mu_{m}, m = \lceil (M+1)/2 \rceil, \dots, M \}$,
	\begin{equation}\nonumber
		\Vert (\pi_{N}^{\mu}-I)\varphi(\mu) \Vert_{\omega^{0,-1}(D)} \lesssim N^{-k} \Vert \varphi(\mu) \Vert_{H^{k}(D)}, \quad	\| \partial_x (\pi_{N}^{\mu}-I)\varphi(\mu) \|_{\omega^{1,0}(D)} \lesssim N^{1-k} \Vert \varphi(\mu) \Vert_{H^{k}(D)},
	\end{equation}
	for $\mu \in \{ \mu_{m}, m = 0, 1, \dots, \lfloor (M-1)/2 \rfloor \}$,
	\begin{equation*}
		\Vert (\pi_{N}^{\mu}-I)\varphi(\mu) \Vert_{\omega^{-1,0}(D)} \lesssim N^{-k} \Vert \varphi(\mu) \Vert_{H^{k}(D)}, \quad	\| \partial_x (\pi_{N}^{\mu}-I)\varphi(\mu) \|_{\omega^{0,1}(D)} \lesssim N^{1-k} \Vert \varphi(\mu) \Vert_{H^{k}(D)}.
	\end{equation*}
	From Lemma \ref{boundary error}, it is easy to see that 
	\begin{equation*}
		\begin{aligned}
			| (\pi_{N}\varphi - \varphi)(-1,\mu) |^{2}+| (\pi_{N}\varphi - \varphi)(1,\mu) |^{2} \lesssim N^{1-2k} \Vert \varphi(\mu) \Vert_{H^{k}(D)}^{2}.
		\end{aligned}
	\end{equation*}
	Thus we derive the estimate of the last term of \eqref{2st}:
	\begin{equation}\label{Value inequality}
		\begin{aligned}
			\sum_{i=0}^{M}&\omega_{i} |\mu_{i}| \,|[ I_{M}^{\mu}(I - \pi_{N})\varphi] (\sign(\mu_{i}),\mu_{i})|^{2} \\ 
			&\lesssim \int_{\Lambda} \Big(| [(I_{M}^{\mu}-I)(\pi_{N}-I)\varphi] (-1,\mu)|^{2} + | [(\pi_{N}-I)\varphi](-1,\mu)|^{2} \\
			&\quad + | [(I_{M}^{\mu}-I)(\pi_{N}-I)\varphi](1,\mu)|^{2} + | [(\pi_{N}-I)\varphi](1,\mu)|^{2} \Big) d\mu \\	
			&\lesssim M^{-2p}N^{-1} \| \varphi \|_{H^{p}(\Lambda;H^{1}(D))}^{2}	+ N^{1-2k} \| \varphi \|_{L^{2}(\Lambda;H^{k}(D))}^{2}.
		\end{aligned}
	\end{equation}
	
	Finally, we deduce from \eqref{inequality 1} to \eqref{Value inequality} immediately that for $\sigma', \gamma',p,q,k,\sigma,\gamma \geq 1$, and $M$ being sufficiently large,
	\begin{equation}\label{eN}
		\begin{aligned}
			& c\| \varphi - \varphi_{N}^{M} \|_{L^{2}(\Lambda;L^{2}(D))}^2 + \sum_{i=0 }^{M}\omega_{i}|\mu_{i}|\big( [\varphi - \varphi_{N}^{M}] (\sign(\mu_{i}),\mu_{i})\big)^{2}  \\ 
			&\lesssim c(M^{-2\sigma} \| \varphi \|_{H^{\sigma}(\Lambda;L^2(D))}^2 + M^{-2p}N^{-2} \| \varphi \|_{H^{p}(\Lambda;H^{1}(D))}^2 + N^{-2k} \| \varphi \|_{L^{2}(\Lambda;H^{k}(D))}^2 )  \\
			& \ \, +N^{1-2k} \| \varphi \|_{L^{2}(\Lambda;H^{k}(D))}^2 + M^{-2p}N^{-1}\| \varphi \|_{H^{p}(\Lambda;H^{1}(D))}^{2} + \frac1c \Big( M^{-2p}\| \varphi \|_{H^{p}(\Lambda;H^{1}(D))}^{2} \\
			&\ \, + M^{-2\sigma} \| \Sigma_{t} \|_{H^{\sigma}(\Lambda;L^2(D))}^{2} \| \varphi \|_{H^{\sigma}(\Lambda;L^2(D))}^{2} + M^{-2\sigma} \| \Sigma_{s} \|_{  H^{\sigma}(\Lambda_{\nu};L^2(\Lambda_{\mu};L^{2}(D)))  }^{2} \| \varphi \|_{H^{\sigma}(\Lambda;L^{2}(D))}^{2} \\
			& \ \, + M^{-2(\sigma'+\gamma')} \| \Sigma_{s} \|_{  H^{\sigma'}(\Lambda_{\nu};H^{\gamma'}(\Lambda_{\mu};L^{2}(D)))  }^{2} \| \varphi \|_{H^{\sigma'}(\Lambda;L^{2}(D))}^{2} \\
			& \ \,
			+ M^{-2\gamma} \| \Sigma_{s} \|_{L^{2}(\Lambda_{\nu};H^{\gamma}(\Lambda_{\mu};L^{2}(D)))}^{2} \| \varphi \|_{L^{2}(\Lambda;L^{2}(D))}^{2} + M^{-2q} \| s \|_{H^{q}(\Lambda;L^{2}(D))}^{2}  \Big). 
		\end{aligned}
	\end{equation} 
	We are now in a position to provide the final error estimate based on the above analysis, Lemma \ref{error} and \eqref{interperror}. 
	\begin{theorem}\label{thm1}
		Let $\varphi$ and $\varphi_{N}^{M}$ be the solutions of \eqref{steady transport} and \eqref{Galerkin 1D} with the boundary condition \eqref{bc}, respectively, $M$ is sufficiently large. When $\Sigma_{t} \in H^{\sigma}(\Lambda;L^2(D))$ $\cap$ $ C(\Lambda;L^{\infty}(D))$, $\Sigma_{s} \in H^{\sigma}(\Lambda_{\nu};L^2(\Lambda_{\mu};L^{2}(D)))$ $\cap$ $L^{2}(\Lambda_{\nu};H^{\gamma}(\Lambda_{\mu};L^{2}(D)))$ $\cap$ $ H^{\sigma'}(\Lambda_{\nu};H^{\gamma'}(\Lambda_{\mu};L^{2}(D)))$, $s \in H^{q}(\Lambda;L^{2}(D))$, and under the assumptions \eqref{hypothesis 1} and \eqref{hypothesis 2}, if $\varphi \in H^{\max\{\sigma,\sigma'\}}(\Lambda;L^{2}(D))$ $\cap$ $L^{2}(\Lambda;H^{k}(D))$ $\cap$ $H^{p}(\Lambda;H^{1}(D))$, $p,q,k,\sigma,\gamma,\sigma',\gamma' \geq 1$, we have 
		\begin{equation}\label{ERROR}
			\begin{aligned}
				\sqrt{c} \| &\varphi - \varphi_{N}^{M} \|_{L^{2}(\Lambda;L^{2}(D))} + \Big( \sum_{i=0 }^{M}\omega_{i}|\mu_{i}|\big( [\varphi - \varphi_{N}^{M}] (\sign(\mu_{i}),\mu_{i})\big)^{2} \Big)^{1/2} \\ 
				& \lesssim N^{1/2-k} \| \varphi \|_{L^{2}(\Lambda;H^{k}(D))} + \frac{1}{\sqrt{c}} \Big( M^{-p}\| \varphi \|_{H^{p}(\Lambda;H^{1}(D))} \\
				&\quad  + \big( \| \Sigma_{t} \|_{H^{\sigma}(\Lambda;L^2(D))} + \| \Sigma_{s} \|_{  H^{\sigma}(\Lambda_{\nu};L^2(\Lambda_{\mu};L^{2}(D)))  } + c \big) M^{-\sigma} \| \varphi \|_{H^{\sigma}(\Lambda;L^{2}(D))}  \\
				& \quad  + M^{-(\sigma'+\gamma')} \| \Sigma_{s} \|_{  H^{\sigma'}(\Lambda_{\nu};H^{\gamma'}(\Lambda_{\mu};L^{2}(D)))  } \| \varphi \|_{H^{\sigma'}(\Lambda;L^{2}(D))}   \\
				&\quad + M^{-\gamma} \| \Sigma_{s} \|_{L^{2}(\Lambda_{\nu};H^{\gamma}(\Lambda_{\mu};L^{2}(D)))} \| \varphi \|_{L^{2}(\Lambda;L^{2}(D))} + M^{-q} \| s \|_{H^{q}(\Lambda;L^{2}(D))} \Big), 
			\end{aligned}
		\end{equation} 
		where $c$ is the constant in assumptions \eqref{hypothesis 1} and \eqref{hypothesis 2}.
	\end{theorem}
	\begin{remark}
		The first term, $ N^{1/2-k} \| \varphi \|_{L^{2}(\Lambda;H^{k}(D))} $, on the right side of \eqref{ERROR} indicates that our estimate is  suboptimal in the space. However, as demonstrated in the numerical example 7 in section \ref{5.7}, this estimate is sharp.
		The second term, $M^{-p}\| \varphi \|_{H^{p}(\Lambda;H^{1}(D))} $, reflects the requirement for the regularity of the projection.
	\end{remark}
	
	\section{Numerical results}\label{sec5}
	In this section, we present numerical scalar flux results and $L^{2}$-error results for various model problems in one-dimensional settings, aiming to showcase the accuracy of the spectral method. The scalar flux is defined as the integral of $\varphi$ over all directions: 
	$$u(x) = \int_{\Lambda} \varphi(x, \mu)d\mu,\quad u_{N}^M(x) = \sum_{i = 0}^{M}\omega_{i} \varphi_{N}^{M}(x, \mu_{i}), \quad x \in D.$$
	Firstly, we consider two one-dimensional slab problems, both of which have exact solutions expressed explicitly. 
	\subsection{Example 1}
	In this example, we thoroughly examine the convergence accuracy of the proposed spectral method by considering a slab with vacuum boundary condition. The problem is characterized by the following specifications:
	$$D = (0, 1), \quad \Sigma_{t} = 1 ,\quad \Sigma_{s} = 0.2 ,\quad s(x,\mu) = 2[(3x^2-12x^3+15x^4-6x^5) \mu ] + 2(\Sigma_{t}-\Sigma_{s})x^3(1-x)^3,$$
	The exact solution of the linear transport equation is 
	$$\varphi(x,\mu) = x^3(1-x)^3.$$
	In this study, we employ a fully spectral discretization scheme to solve the equation. 
	The numerical results are presented in Figure \ref{ex4N20}, where the solid line represents the flux of the exact solution for Example 1, and the scattered points represent the numerical flux of our spectral method
	by using a polynomial  of  degree $N=10$ in spatial space  and   of degree $M = 11$ in angular space.
	\begin{figure}[H]
		\centering
		\subfigure[Numerical flux with $N=10$ and $M=11$ and exact flux.]{
			\begin{minipage}[t]{0.45\linewidth}
				\centering
				\includegraphics[width=\linewidth]{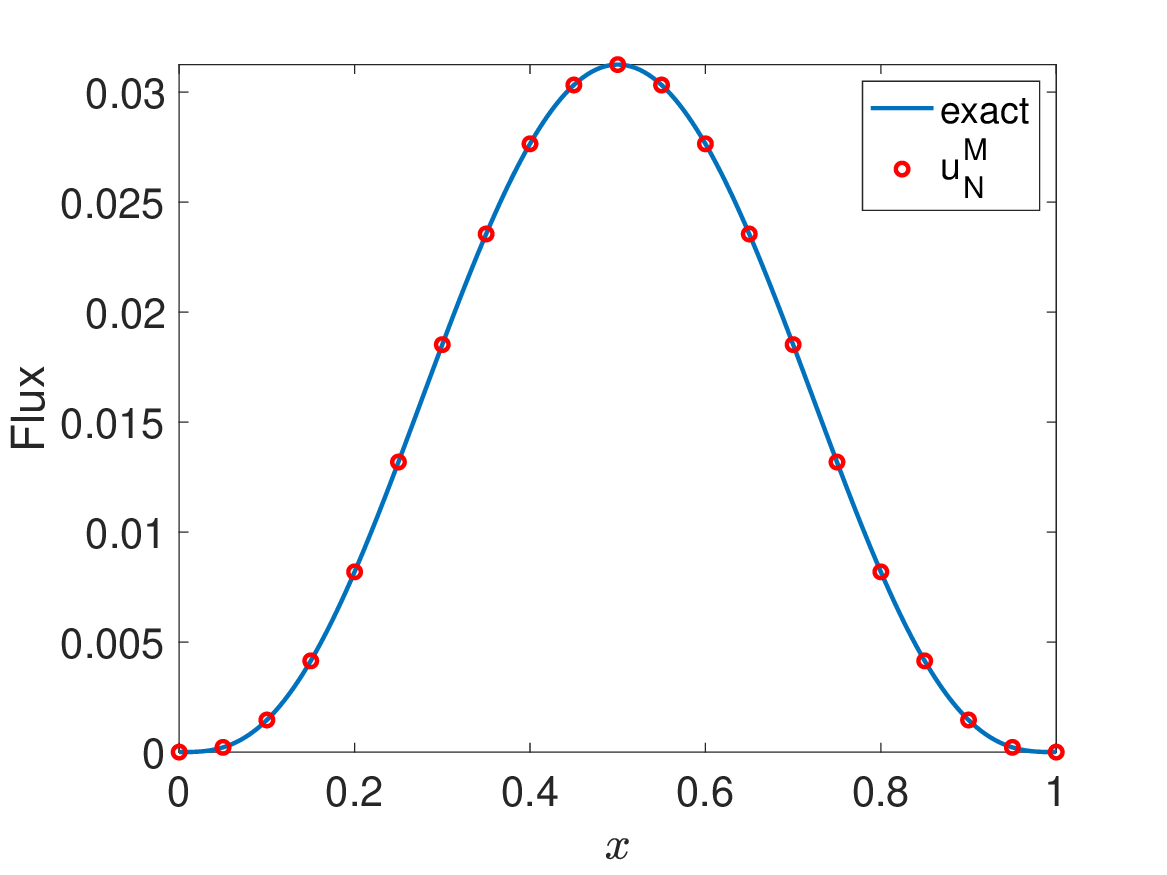}
				\label{ex4N20}
			\end{minipage}
		}
		\hspace{0.6em}
		\subfigure[$L^1$-errors of flux evaluated by the HWENO method.]{
			\begin{minipage}[t]{0.45\linewidth}
				\centering
				\includegraphics[width=\linewidth]{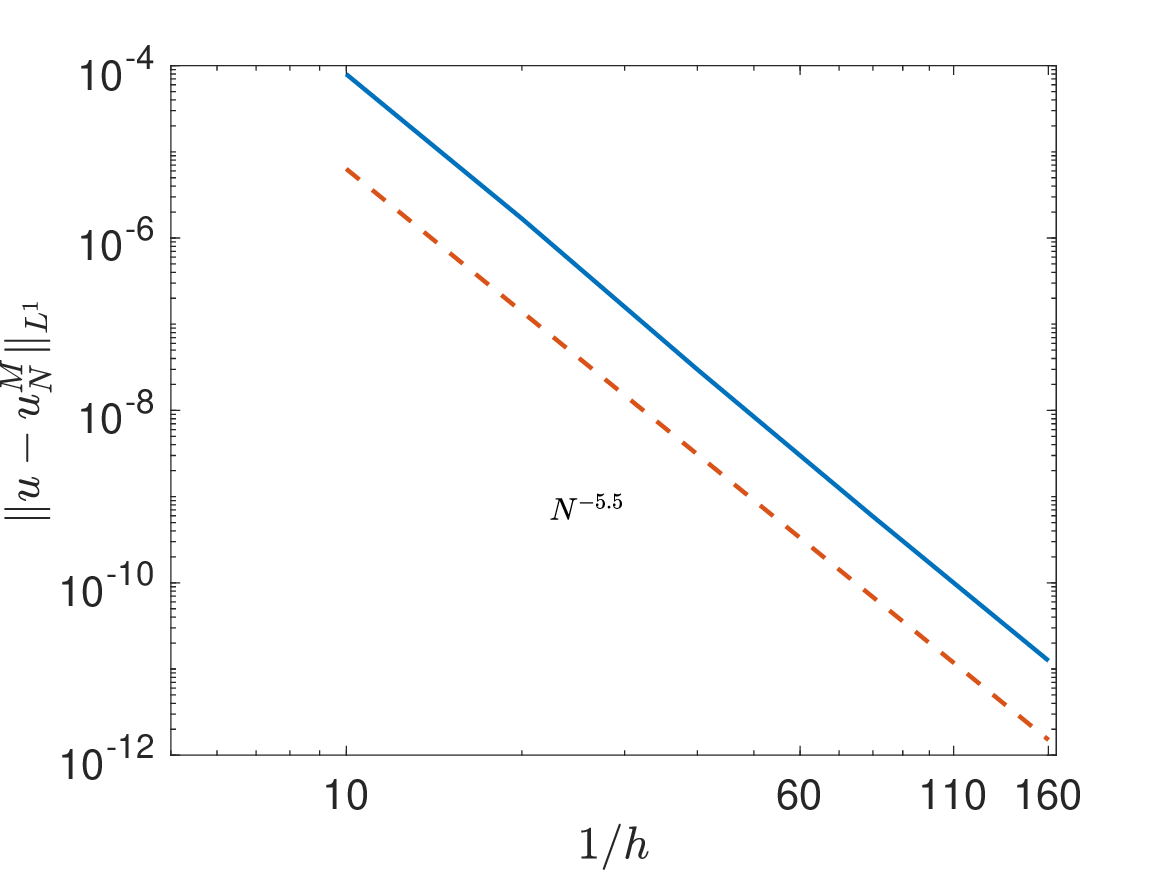} 
				\label{errorex1HWENO}
			\end{minipage}
		}
		\caption{The flux of Example 1 for spectral method and $L^1$-errors of flux evaluated by the HWENO method.}
		\label{figex4}
	\end{figure}
	While Figure \ref{errorex1HWENO} illustrates the relationship between the $L^1$-errors and $1/h$ for Example 1 using the high order HWENO (Hermite weighted essentially non-oscillatory) method with $M = 11$, where $h$ is the spatial size. The error data,  which is obtained from the research work presented in \cite{ren2022high},  roughly exhibits a convergence order  of $\mathcal{O}(h^{5.5})$.
	
	The results demonstrate the convergence of the proposed spectral method in approximating the neutron transport equation. The scattered points closely align with the solid line, indicating that the numerical solution effectively converges to the exact solution as the degree of the polynomial increases.
	
	We present the $L^2$-errors between the numerical flux of Example 1 and the exact solution flux in Figure \ref{errorex4}. Figure \ref{errorex4N} illustrates the $L^2$-errors with respect to the spatial approximation polynomial degree $N$ at $M=11$, while Figure \ref{errorex4M} displays the $L^2$-errors with respect to the angular discretization number $M+1$ at $N=30$.
	\begin{figure}[H]
		\centering
		\subfigure[Numerical errors $\Vert u-u_{N}^{M} \Vert_{L^{2}}$ vs. $N$ $(M=11)$.]{
			\begin{minipage}[t]{0.45\linewidth}
				\centering
				\includegraphics[width=\linewidth]{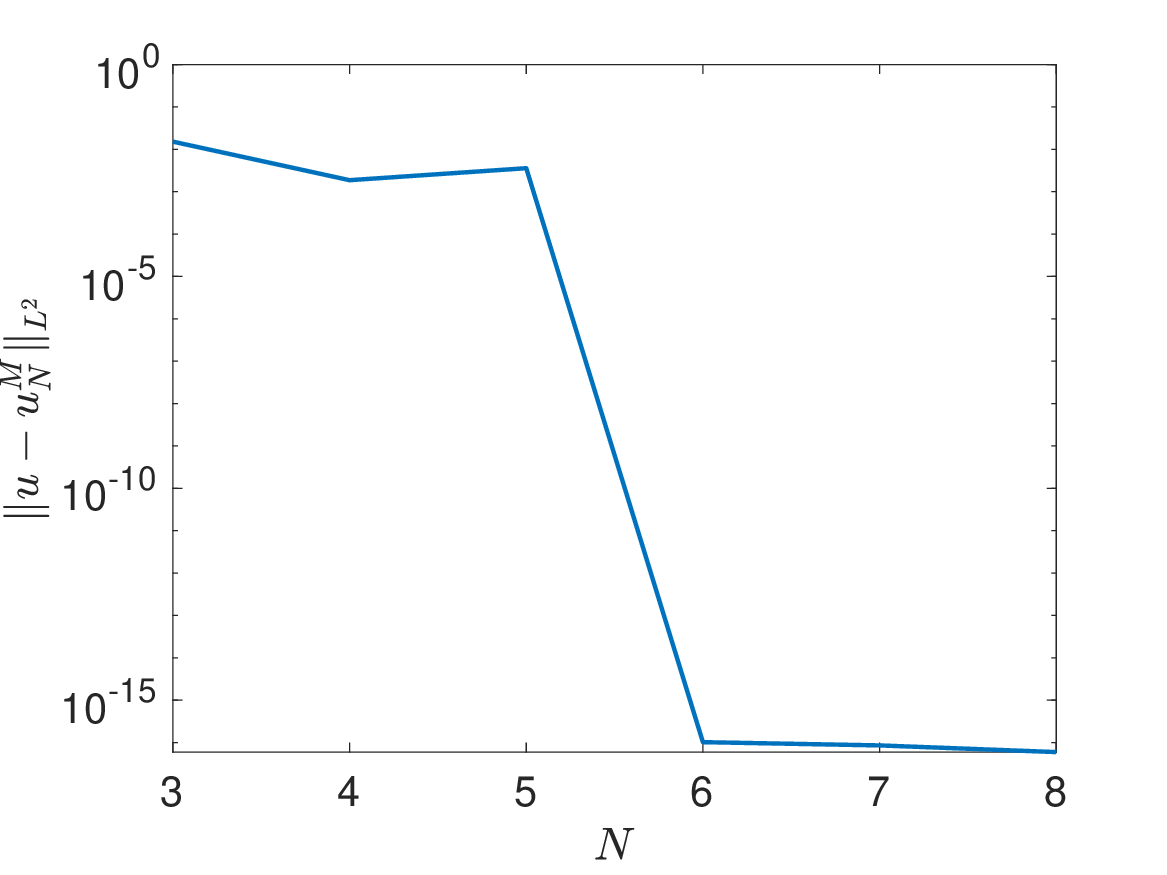}
				\label{errorex4N}
			\end{minipage}
		}
		\hspace{0.6em}
		\subfigure[Numerical errors $\Vert u-u_{N}^{M} \Vert_{L^{2}}$ vs. $M+1$ $(N=30)$.]{
			\begin{minipage}[t]{0.45\linewidth}
				\centering
				\includegraphics[width=\linewidth]{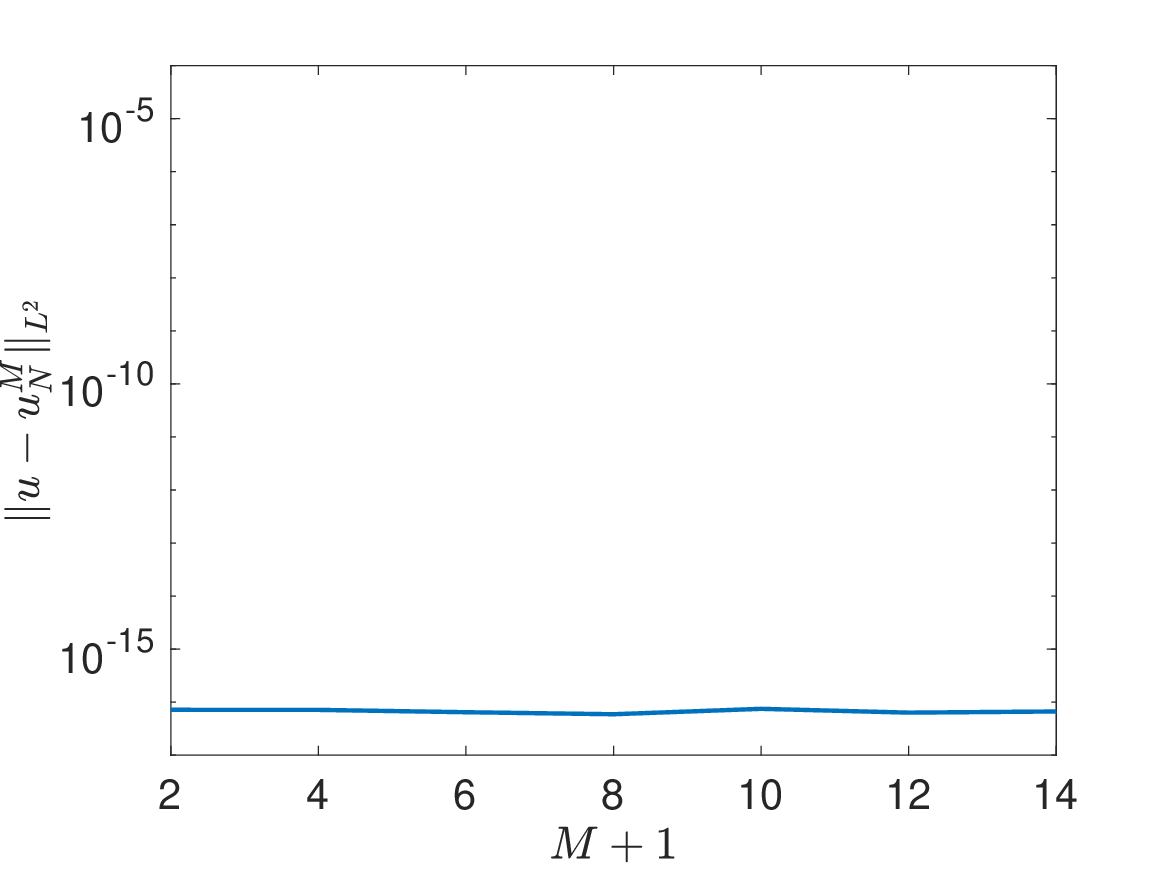}
				\label{errorex4M}
			\end{minipage}
		}
		\caption{The $L^{2}$-errors versus $N$ and $M+1$ of Example 1 using the spectral method.}
		\label{errorex4}
	\end{figure}
	From Figure \ref{errorex4N}, it is evident that the proposed method exhibits spectral accuracy concerning the polynomial degree $N$. As the degree $N$ increases to 6, the $L^2$-errors approach the machine epsilon, indicating the convergence of the numerical solution to the exact solution. Similarly, in Figure \ref{errorex4M}, it can be observed that when the angular discretization number $M+1$ is small (e.g., $M=1$), the $L^2$-errors also approach the machine epsilon. This occurs because the exact solution is a sixth-degree polynomial solely in the spatial variable $x$ and is entirely independent of the angular variable $\mu$. As a result, selecting $N = 6$ and $M = 1$ enables the $L^2$-error to reach machine epsilon.
	
	Furthermore, by comparing Figure \ref{errorex1HWENO} and Figure \ref{errorex4N}, it is evident that for problems with a higher regularity of the exact solution, the spectral method requires fewer degrees of freedom compared to the Hermite WENO fast sweeping method to achieve the same level of accuracy.
	\subsection{Example 2}
	In this test, we solve the absorbing-scattering transfer problem described by the equation \eqref{steady transport} with 
	$$\Sigma_{t} = 22000, \quad \Sigma_{s} = 1,$$ 
	$$s(x,\mu) = -4 \pi \mu^{3} \cos^{3}\pi x \sin\pi x + \Sigma_{t}(\mu^{2}\cos^{4}\pi x + const) - \Sigma_{s}(const + \frac{\cos^{4}\pi x}{3}).$$
	Here $const = 10^{-14}$ is a small positive constant which is used to ensure the source term to be nonnegative. The computational domain is $D = (0,1)$. The boundary condition is given as follows 
	\begin{equation*}
		\left\{
		\begin{aligned}
			\varphi(0,\mu) &= \mu^{2} + const, \quad \text{if } \mu > 0, \\
			\varphi(1,\mu) &= \mu^{2} + const, \quad \text{if } \mu < 0.
		\end{aligned}
		\right. 
	\end{equation*}
	For this problem, we have the exact solution given as \cite{yuan2016high} 
	$$\varphi(x,\mu) = \mu^{2}\cos^{4}\pi x + const.$$
	
	\begin{figure}[H]
		\centering
		\subfigure[Flux with $N=10$ and $M=11$.]{
			\begin{minipage}[t]{0.45\linewidth}
				\centering
				\includegraphics[width=\linewidth]{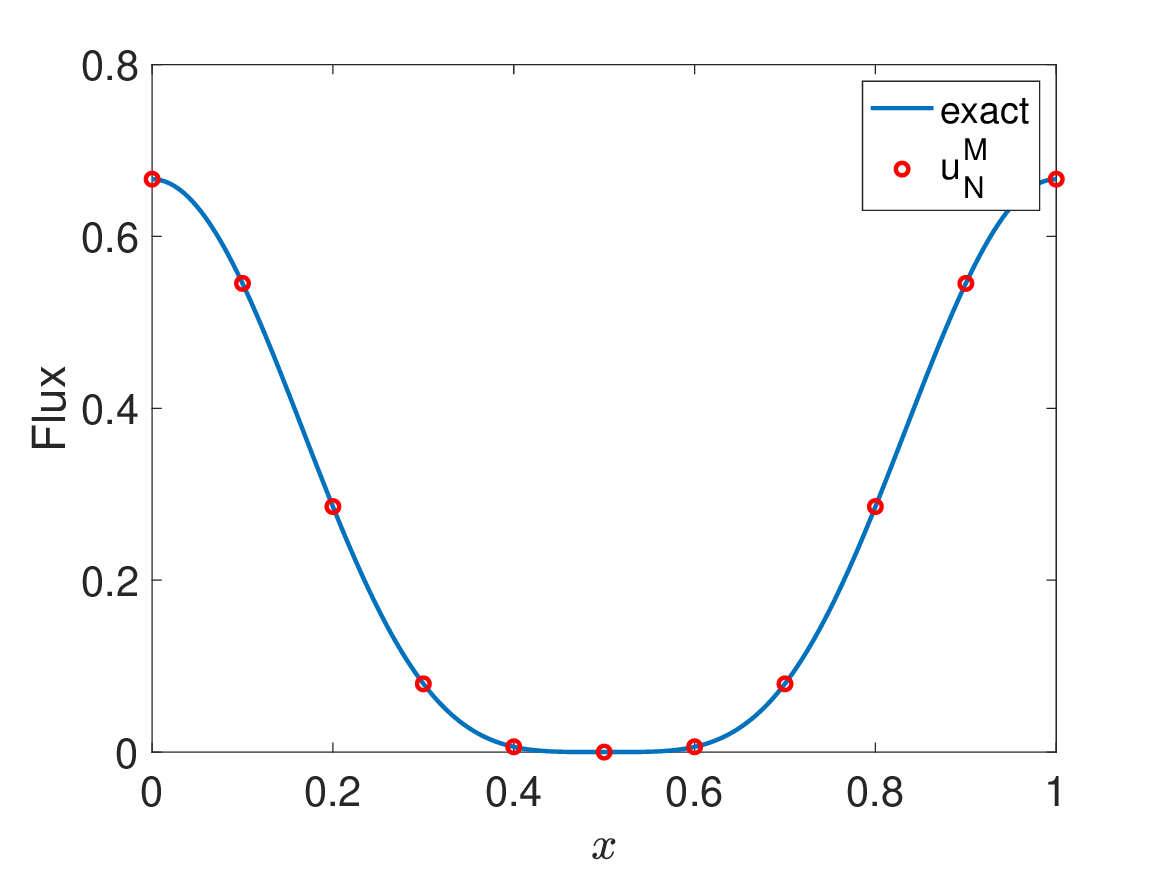}
				\label{ex3N10}
			\end{minipage}
		}
		\subfigure[$L^2$-errors of the $P_1$ DG scheme with the positivity-preserving limiter.]{
			\begin{minipage}[t]{0.45\linewidth}
				\centering
				\includegraphics[width=\linewidth]{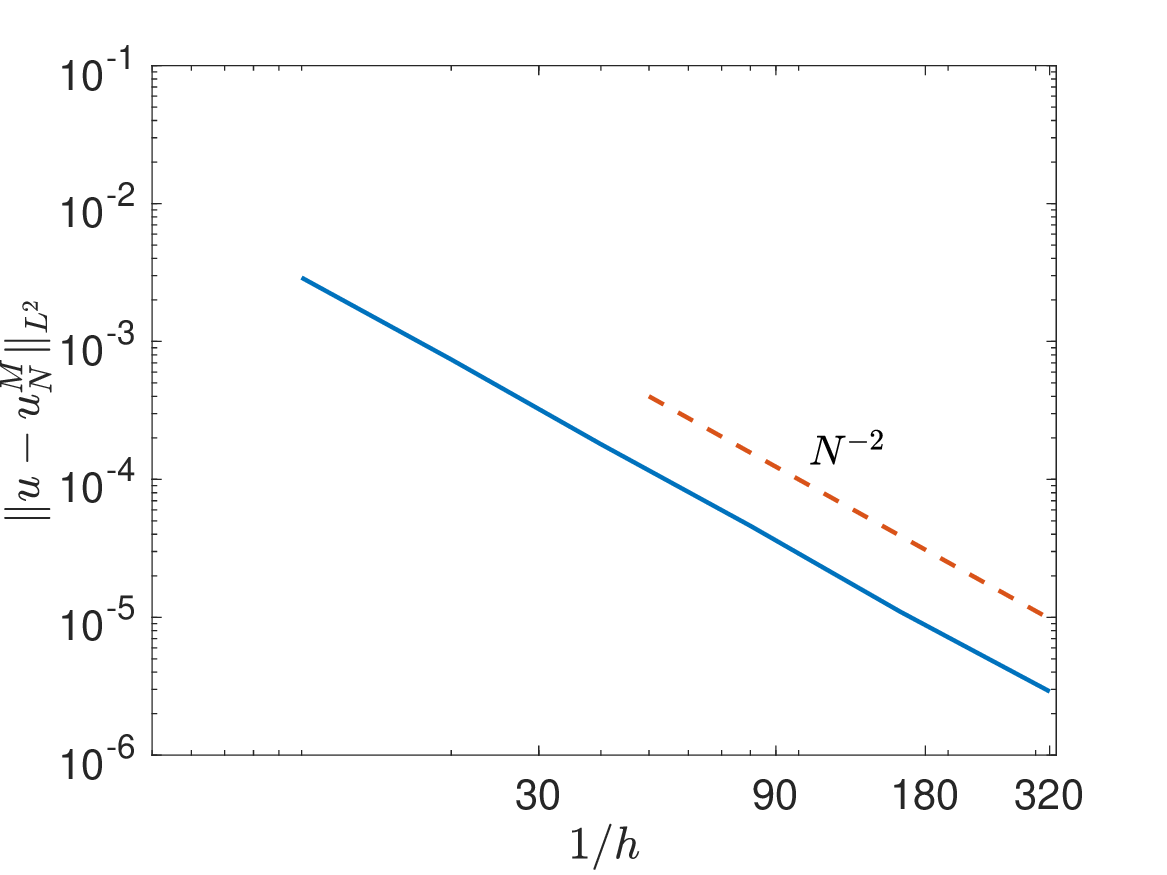}
				\label{ex3DG}
			\end{minipage}
		}
		\caption{The flux of Example 2 obtained using the spectral method and $L^2$-errors of $P_{1}$ DG scheme with the positivity-preserving limiter.}
		\label{figex3}
	\end{figure}
	\begin{figure}[H]
		\centering
		\subfigure[Numerical errors $\Vert u-u_{N}^{M} \Vert_{L^{2}}$ vs. $N$ $(M=11)$.]{
			\begin{minipage}[t]{0.45\linewidth}
				\centering
				\includegraphics[width=\linewidth]{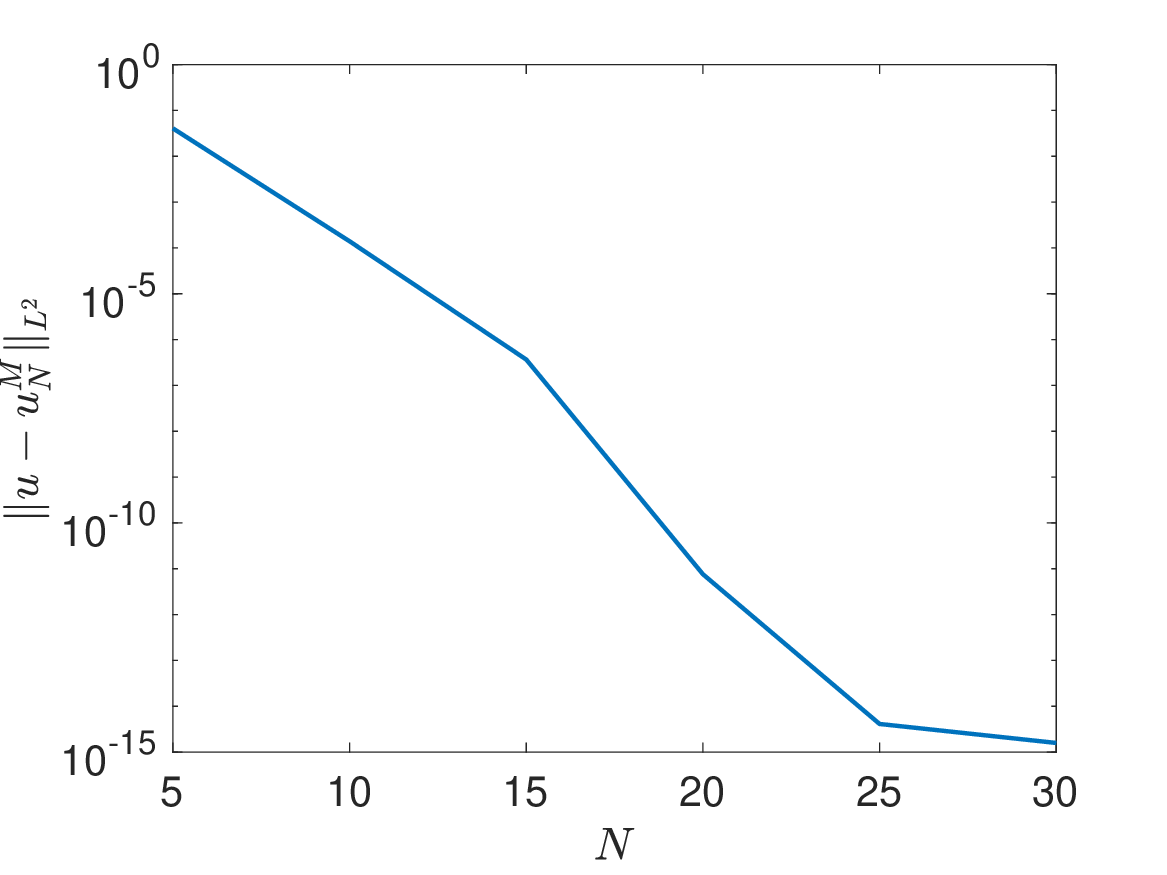}
				\label{errorex3N}
			\end{minipage}
		}
		\subfigure[Numerical errors $\Vert u-u_{N}^{M} \Vert_{L^{2}}$ vs. $M+1$ $(N=30)$.]{
			\begin{minipage}[t]{0.45\linewidth}
				\centering
				\includegraphics[width=\linewidth]{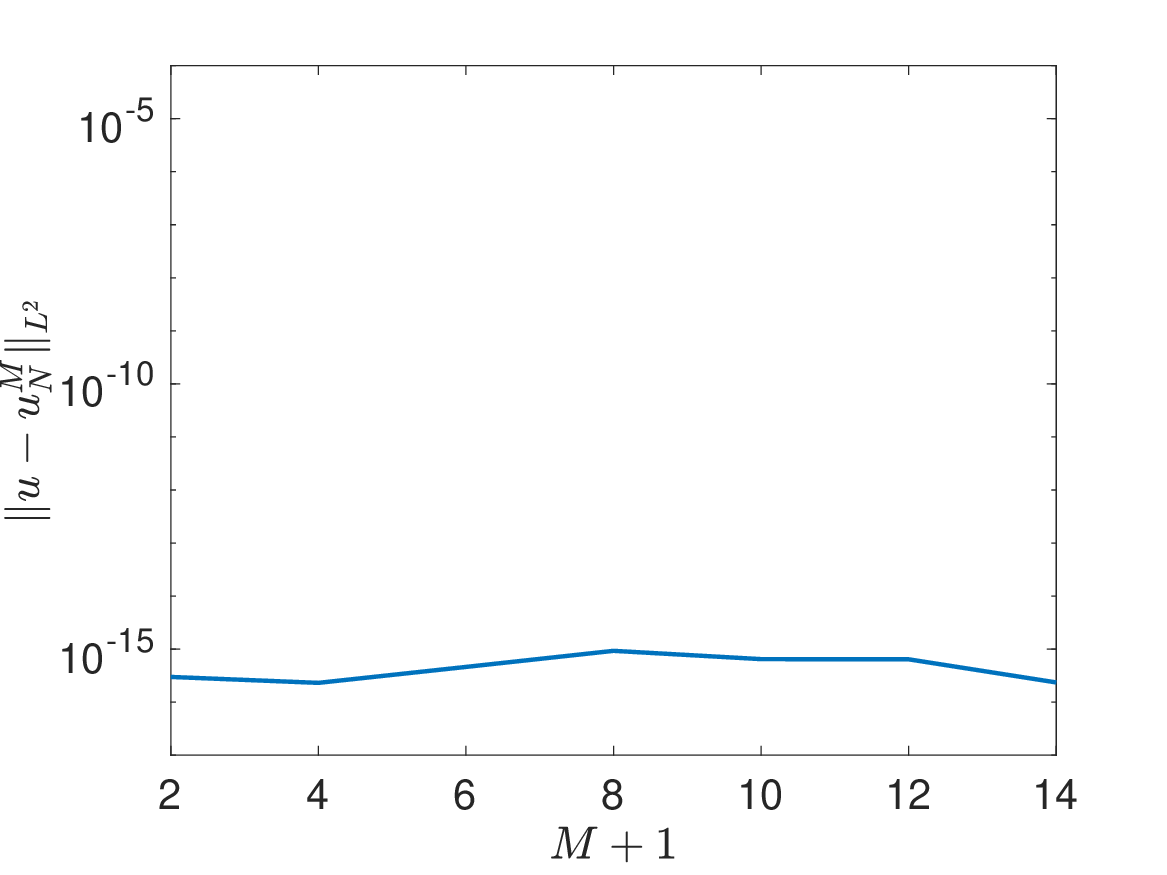}
				\label{errorex3M}
			\end{minipage}
		}
		\caption{The $L^{2}$-errors versus $N$ and $M+1$ of Example 2 for the spectral method.}
		\label{errorex3}
	\end{figure}
	Figure \ref{figex3} consists of two subplots. In the left subplot, the flux of the exact solution and the numerical solution obtained using the spectral method are depicted with the polynomial degree $N = 10$ and the  angular discretization   parameter  $M = 11$. The right subplot presents the $L^2$-errors of the $P_{1}$ DG scheme with the positivity-preserving limiter for Example 2, as obtained from the research in \cite{yuan2016high}.
	
	We present the $L^2$-error between the numerical flux and the exact solution flux of Example 2 in Figure \ref{errorex3}. Figure \ref{errorex3N} displays the $L^2$-errors with respect to the spatial approximation polynomial degree $N$ at $M=11$, while Figure \ref{errorex3M} illustrates the $L^2$-errors with respect to the angular discretization number $M+1$ at $N=30$. 
	
	From Figure \ref{errorex3N}, it is evident that our method exhibits spectral accuracy with respect to the polynomial degree $N$. As $N$ increases to 25, the $L^2$-errors approach the machine epsilon, indicating the convergence of the numerical solution to the exact solution. Similarly, from Figure \ref{errorex3M}, it can be observed that when the angular discretization parameter $M$ is small, the $L^2$-errors also approach the machine epsilon. This is because the exact solution itself is a trigonometric function in the spatial variable $x$ and a polynomial of degree 2 in the direction variable $\mu$, so when $N$ reaches 25 and $M$ reaches 2, the $L^2$-errors reach machine epsilon.
	
	By comparing Figure \ref{ex3DG}  and Figure \ref{errorex3N}, it becomes evident that, under the same degree of freedom, the spectral method achieves higher accuracy than the low order DG scheme.
	\subsection{Example 3}
	Let us consider an example of a non-vacuum boundary condition, the setup of the problem takes the form
	$$D = (0,1), \quad \Sigma_{t}=100, \quad \Sigma_{s}=99.992, \quad s=0.01,$$
	with the boundary condition: 
	$$\varphi(0, \mu)=5-5\mu, \text{ for } \mu>0 \text{ and }\varphi(1, \mu)=0, \text{ for }\mu<0.$$
	%
	\begin{figure}[H]
		\centering
		\subfigure[Flux with $N=20$ and $M=11$.]{
			\begin{minipage}[t]{0.45\linewidth}
				\centering
				\includegraphics[width=\linewidth]{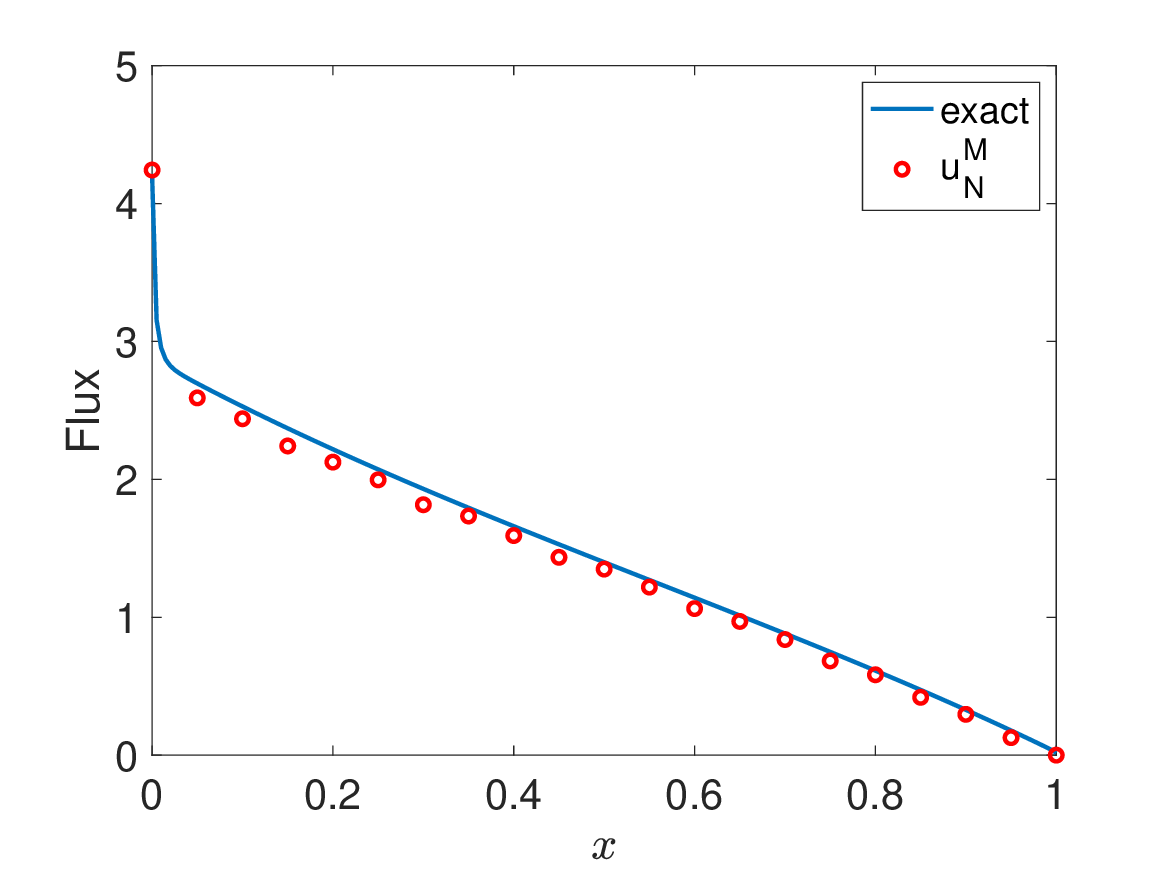}
				\label{1}
			\end{minipage}
		}
		\subfigure[Flux with $N=40$ and $M=11$.]{
			\begin{minipage}[t]{0.45\linewidth}
				\centering
				\includegraphics[width=\linewidth]{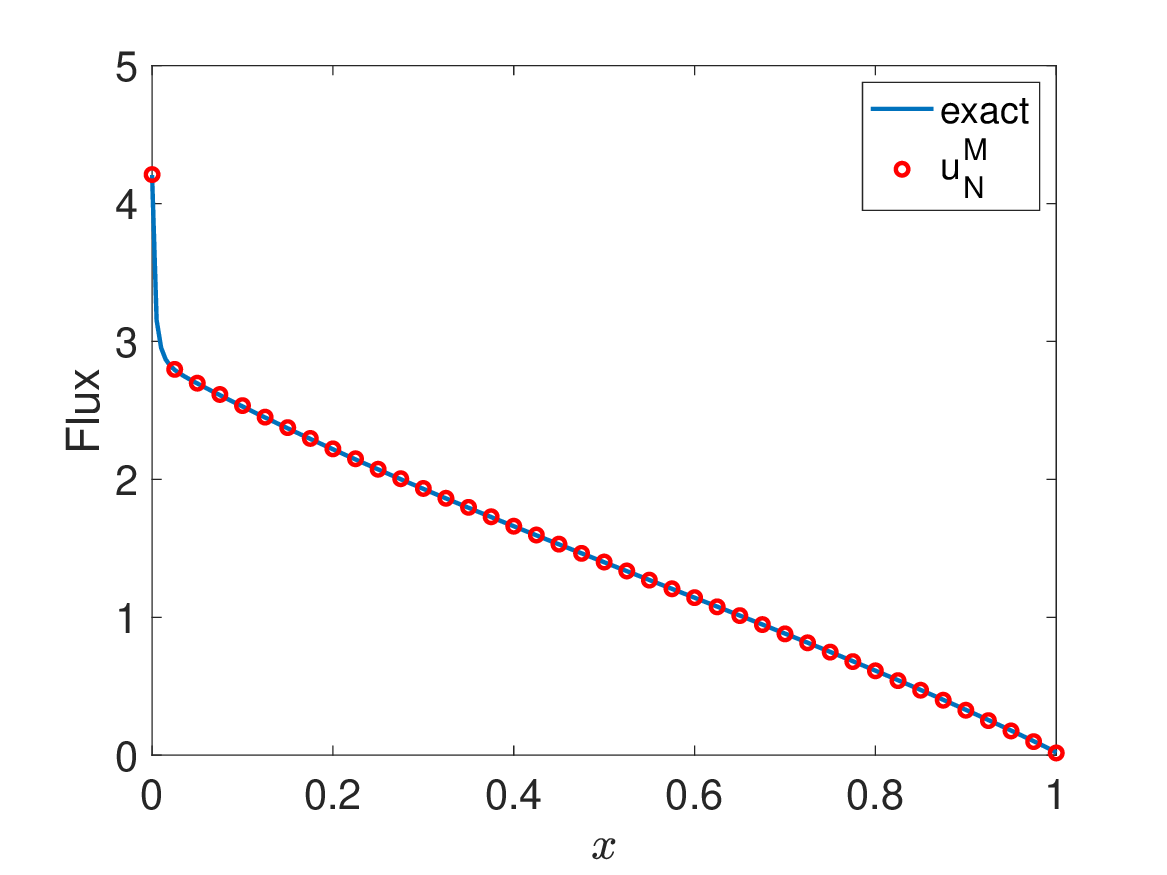}
				\label{2}
			\end{minipage}
		}
		\caption{The fluxes of Example 3 evaluated using the spectral method.}
		\label{fig1}
	\end{figure}
	To approximate Example 3, we utilize a fully spectral approximation scheme with $N$-degree polynomials in spatial space and    $(M+1)$  Legendre-Gauss collocation  points in angular space. The fluxes, presented in Figure \ref{fig1}, exhibit large gradients. The solid line represents the flux of the reference solution $u_{ref}$ for Example 3, corresponding to our fully  spectral  scheme  with $N = 200$ and $M = 11$. The scattered points in the plots represent the flux of the numerical solution. Specifically, Figure \ref{1} displays the flux of the 20th-degree polynomial spectral method with $M = 11$, while Figure \ref{2} shows the flux of the 40th-degree polynomial spectral method with $M = 11$.
	\begin{figure}[H]
		\centering
		\subfigure[Numerical errors $\Vert u-u_{N}^{M} \Vert_{L^{2}}$ vs. $N$ $(M=11)$.]{
			\begin{minipage}[t]{0.45\linewidth}
				\centering
				\includegraphics[width=\linewidth]{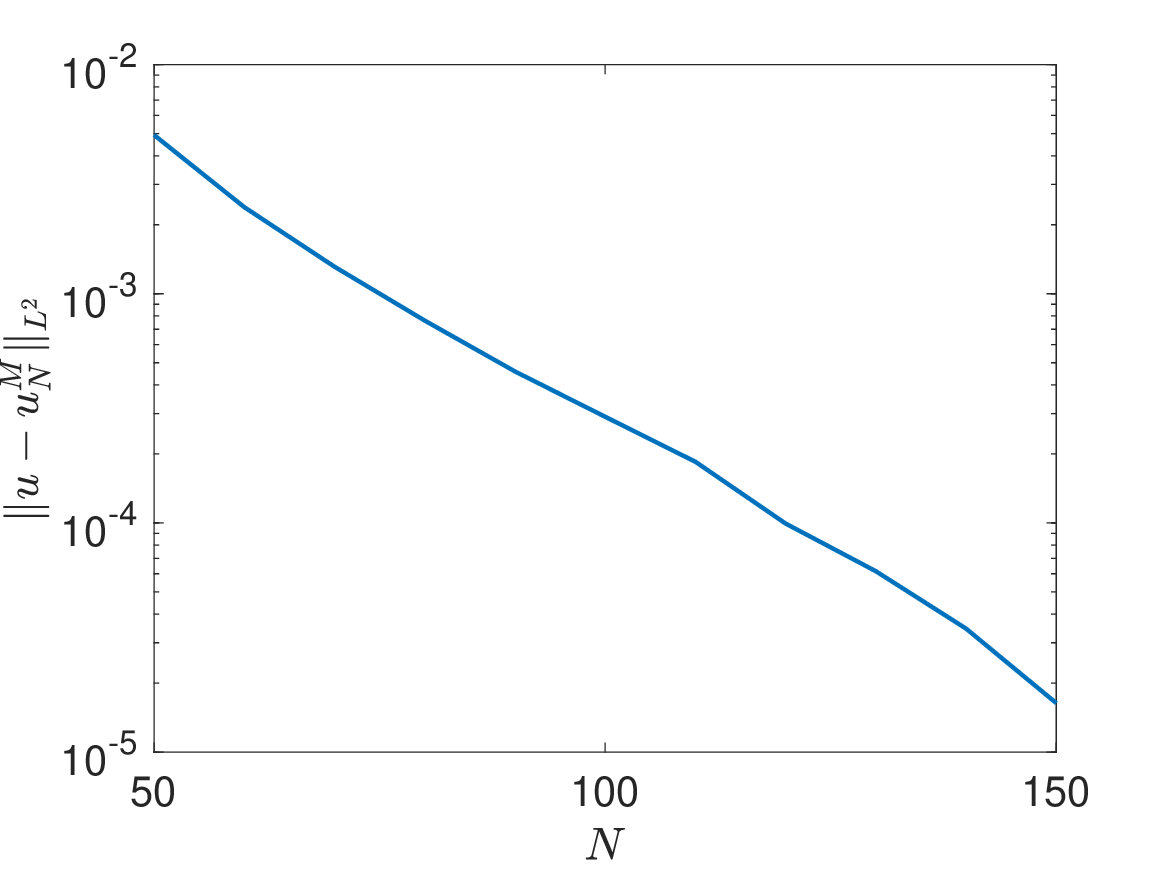}
				\label{error_1N}
			\end{minipage}
		}
		\hspace{0.6em}
		\subfigure[Numerical errors $\Vert u-u_{N}^{M} \Vert_{L^{2}}$ vs. $M+1$ $(N=160)$.]{
			\begin{minipage}[t]{0.45\linewidth}
				\centering
				\includegraphics[width=\linewidth]{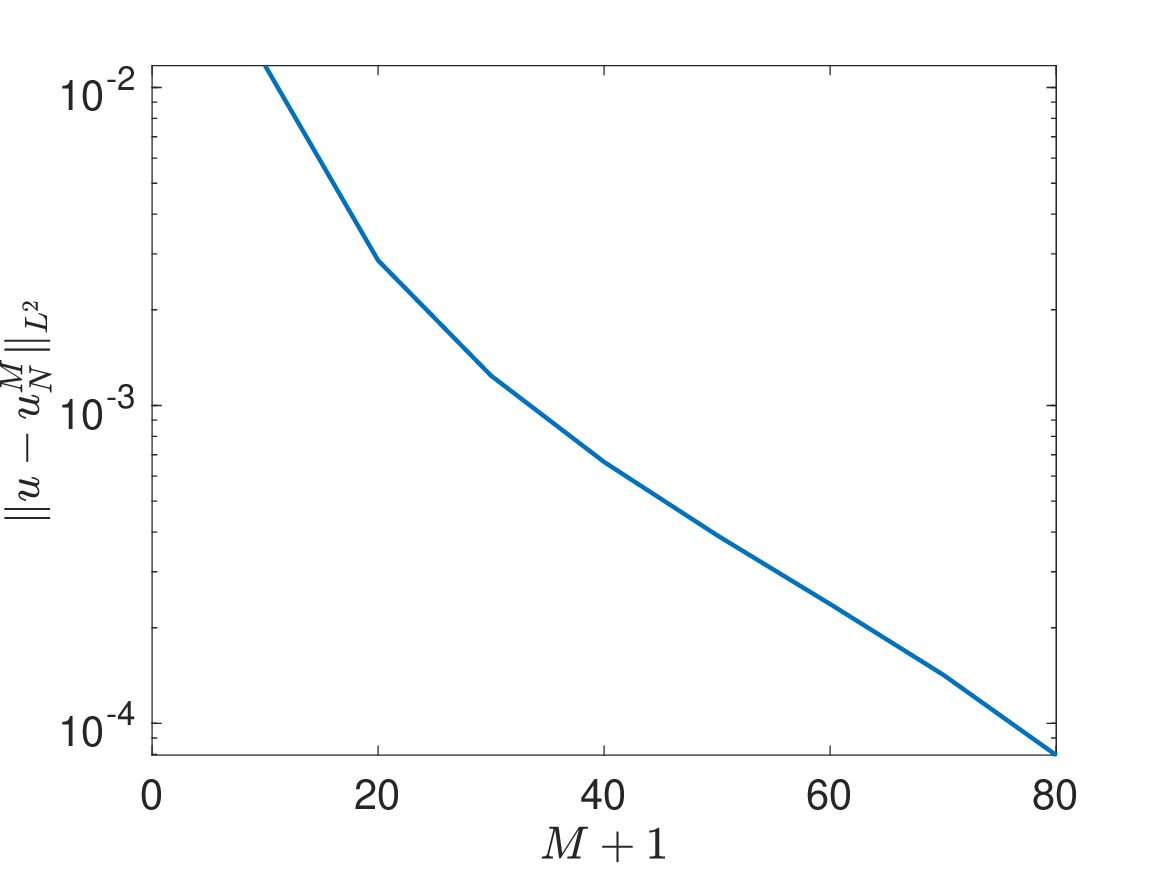}
				\label{error_1M}
			\end{minipage}
		}
		\caption{The $L^{2}$-errors of Example 3 evaluated by the spectral method.}
		\label{error_1}
	\end{figure}
	Figure \ref{error_1N} visually depicts the relationship between the $L^2$-errors of the fluxes for the reference solution and the numerical solution obtained using our spectral method, across various polynomial degree $N$. This plot specifically pertains to the case when $M = 11$. Furthermore, Figure \ref{error_1M} illustrates  the $L^2$-errors between the numerical flux and the reference  flux as $M$ increases, with the fixed  polynomial degree $N = 160$.
	
	Figure \ref{error_1} reveals that the $L^2$-errors of the numerical solution in Example 3 exhibit spectral orders of convergence  concerning both the polynomial degree $N$ and the discrete number $M+1$ in the direction of motion, even for problems with large gradients, consistent with Theorem \ref{thm1}. 
	\subsection{Example 4}
	In this example, we take
	$$D = (0,1), \quad \Sigma_{t}=100,\quad \Sigma_{s}=99.992,\quad s=0.01,$$
	with the vacuum boundary condition on both sides:
	\begin{align*}\nonumber
		\varphi(0, \mu)=0,\quad\mu>0,\quad	\varphi(1, \mu)=0,\quad\mu<0.
	\end{align*}
	The flux for Example 4, obtained using the fully spectral scheme, is presented in Figure \ref{5}.  Specifically, Figure \ref{5} displays the flux result obtained using our spectral method with the  polynomial degrees $N=10$ and  $M = 11$, alongside the reference solution obtained from experimental data in \cite{ren2022high}. Meanwhile Figure \ref{6} illustrates the relationship between the $L^1$-errors and $1/h$ for Example 4 using the HWENO method with the $12$ discrete ordinates  in the angular discretization \cite{ren2022high}.
	\begin{figure}[H]
		\centering
		\subfigure[Flux with $N=10$ and $M=11$.]{
			\begin{minipage}[t]{0.45\linewidth}
				\centering
				\includegraphics[width=\linewidth]{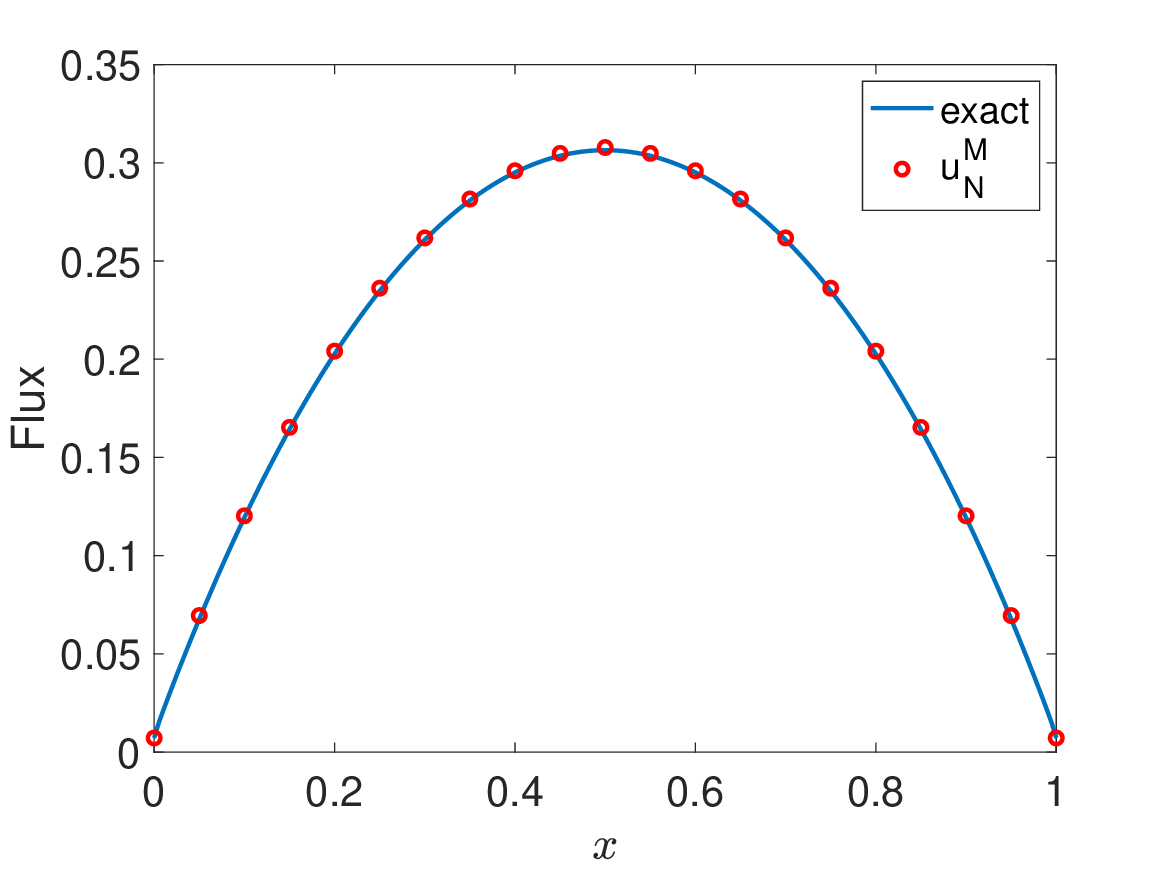}
				\label{5}
			\end{minipage}
		}
		\subfigure[$L^1$-errors of flux obtained by HWENO method.]{
			\begin{minipage}[t]{0.45\linewidth}
				\centering
				\includegraphics[width=\linewidth]{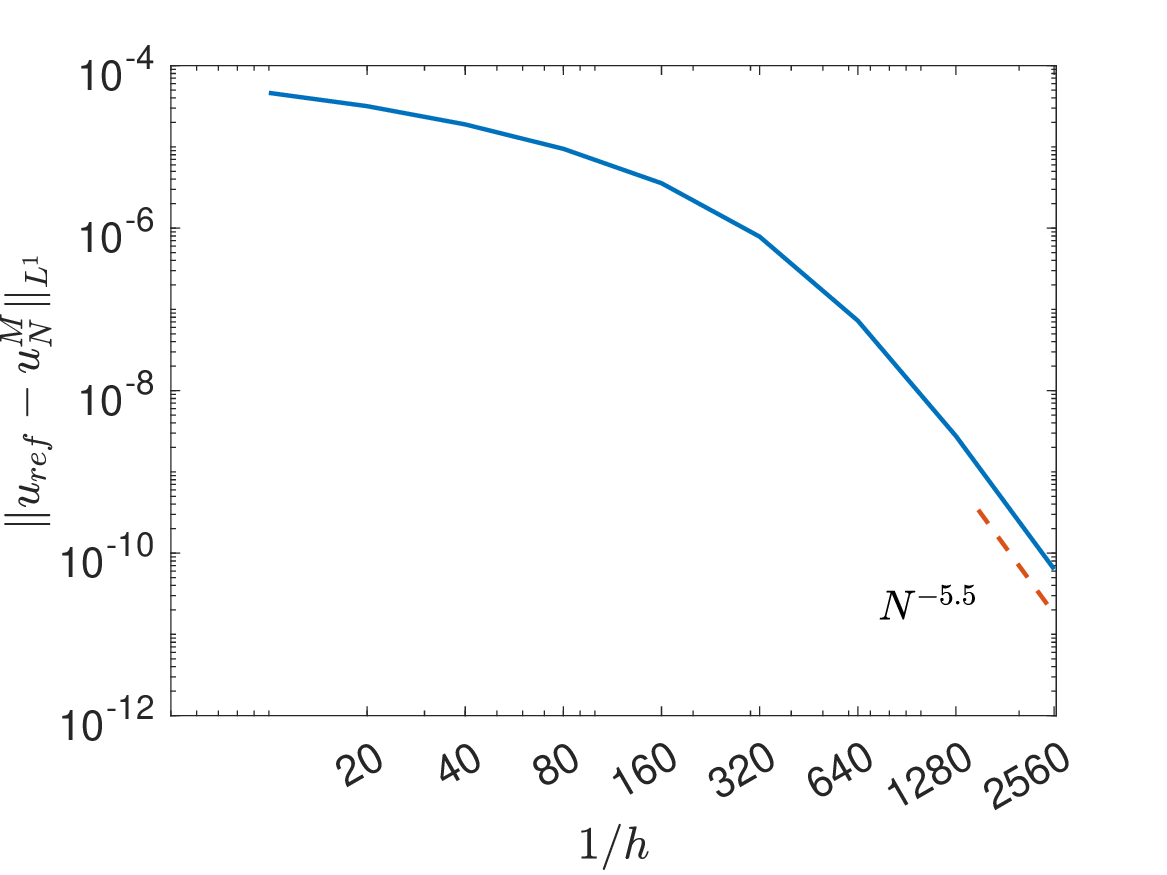}
				\label{6}
			\end{minipage}
		}
		\caption{The flux of Example 4 using the spectral method and $L^1$-errors of Example 4 for HWENO method.}
		\label{fig3}
	\end{figure}
	Figure \ref{fig5} displays the $L^2$-errors of the flux obtained using the spectral method for Example 4, plotted against the polynomial degree $N$ with $M=11$. Additionally, Figure \ref{fig6} showcases the $L^2$-errors of the flux for Example 4 as a function of the number of angular discretization $M+1$, with $N=140$.
	
	Based on the observations in Figure \ref{error2}, it is evident that the $L^2$-errors of the flux in the numerical solution for Example 4 exhibit spectral accuracy with respect to both the polynomial degree $N$ and the discrete number $M+1$ in the direction of motion.
	\begin{figure}[H]
		\centering
		\subfigure[Numerical errors $\Vert u-u_{N}^{M} \Vert_{L^{2}}$ vs. $N$ $(M=11)$.]{
			\begin{minipage}[t]{0.45\linewidth}
				\centering
				\includegraphics[width=\linewidth]{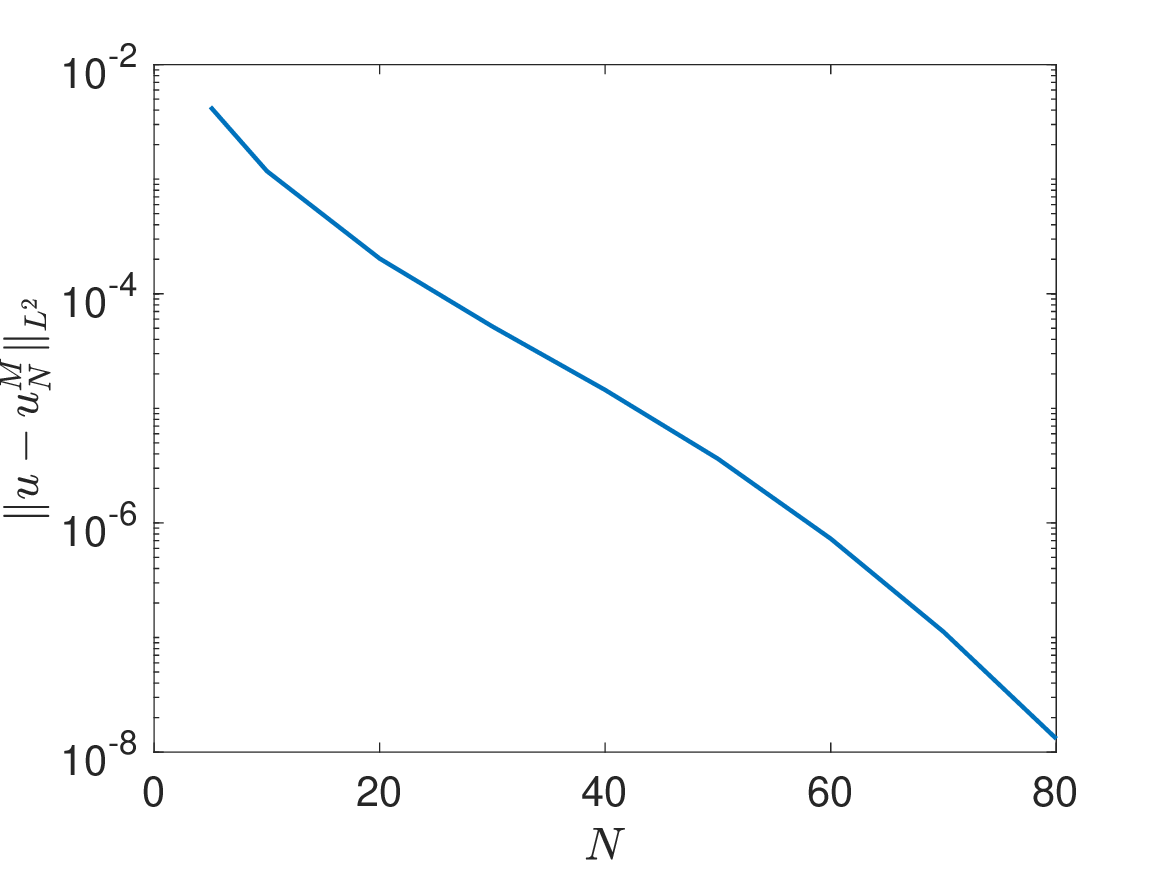}
				\label{fig5}
			\end{minipage}
		}
		\hspace{0.6em}
		\subfigure[Numerical errors $\Vert u-u_{N}^{M} \Vert_{L^{2}}$ vs. $M+1$ $(N=140)$.]{
			\begin{minipage}[t]{0.45\linewidth}
				\centering
				\includegraphics[width=\linewidth]{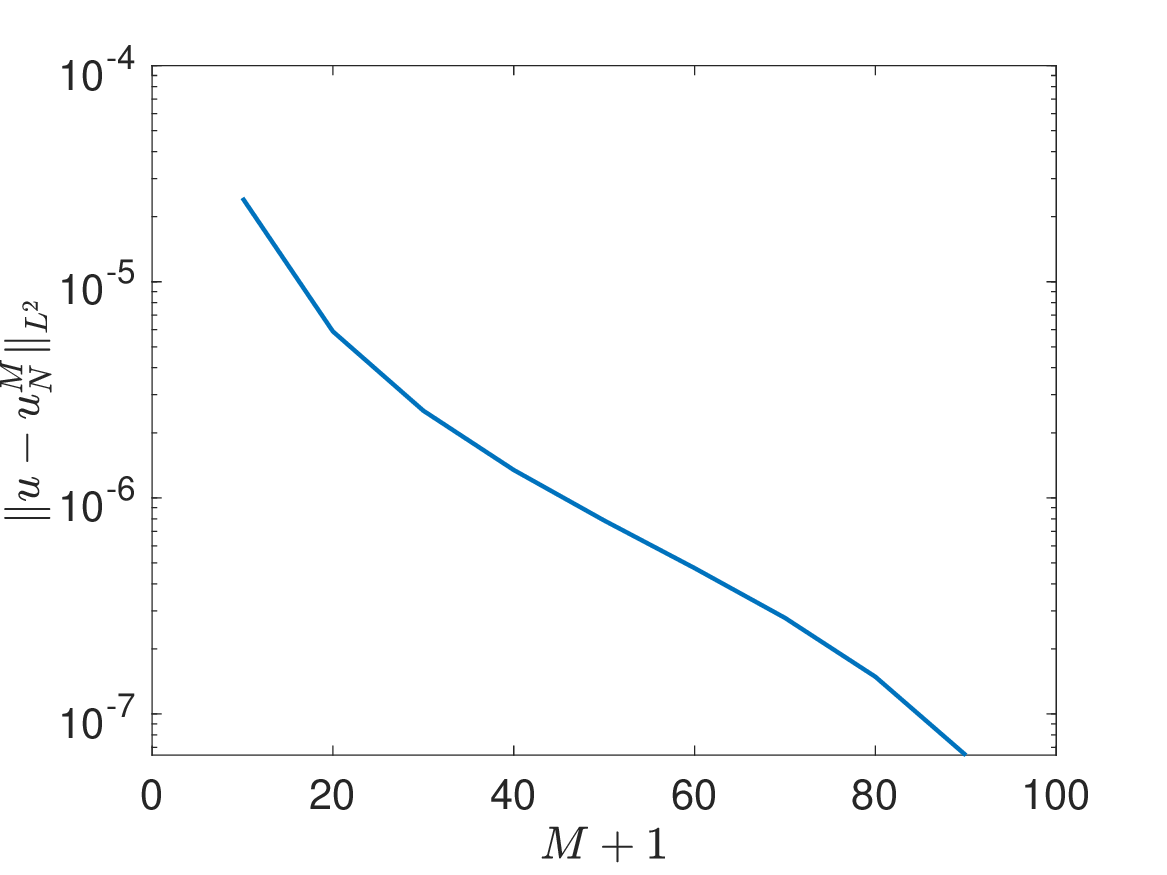}
				\label{fig6}
			\end{minipage}
		}
		\caption{The $L^{2}$-errors of Example 4 obtained using the spectral method.}
		\label{error2}
	\end{figure}
	Comparing Figure \ref{6} and Figure \ref{fig5}, we observe that, as $N$ gradually increases, our spectral method outperforms the HWENO method in terms of accuracy under the same degree of freedom. The spectral method demonstrates superior accuracy and efficiency compared to the HWENO method when solving the same neutron transport equation.
	
	Typically, the neutron transport equation features discontinuous coefficients. Below, we present two examples with such coefficients and utilize the multi-domain spectral method to solve these problems.
	\subsection{Example 5}
	We set $D = (0,2)$, the exact solution $\varphi(x,\mu) = x^3(2-x)^3$, 
	
	\[
	\Sigma_t(x,\mu) = \left\{ 
	\begin{array}{ll}
		3-x,    & 0 \leq x \leq 1, \\[0.5em]	
		x,  & 1 < x \leq 2,	
	\end{array}
	\right. 
	\quad
	\Sigma_s(x,\mu) = \left\{ 
	\begin{array}{ll}
		2-x,    & 0 \leq x \leq 1, \\[0.5em]	
		x-1/2,  & 1 < x \leq 2,	
	\end{array}
	\right. 
	\]
	
	\[
	s(x,\mu) = \left\{ 
	\begin{array}{ll}
		12 x^2(1-x)(2-x)^2 \mu + 2x^3(2-x)^3,    & 0 \leq x \leq 1, \\[0.5em]	
		12 x^2(1-x)(2-x)^2 \mu +  x^3(2-x)^3,  & 1 < x \leq 2.	
	\end{array}
	\right. 
	\]
	\begin{figure}[H]
		\centering
		\subfigure[Flux with $N=10$ and $M+1=2$.]{
			\begin{minipage}[t]{0.45\linewidth}
				\centering
				\includegraphics[width=\linewidth]{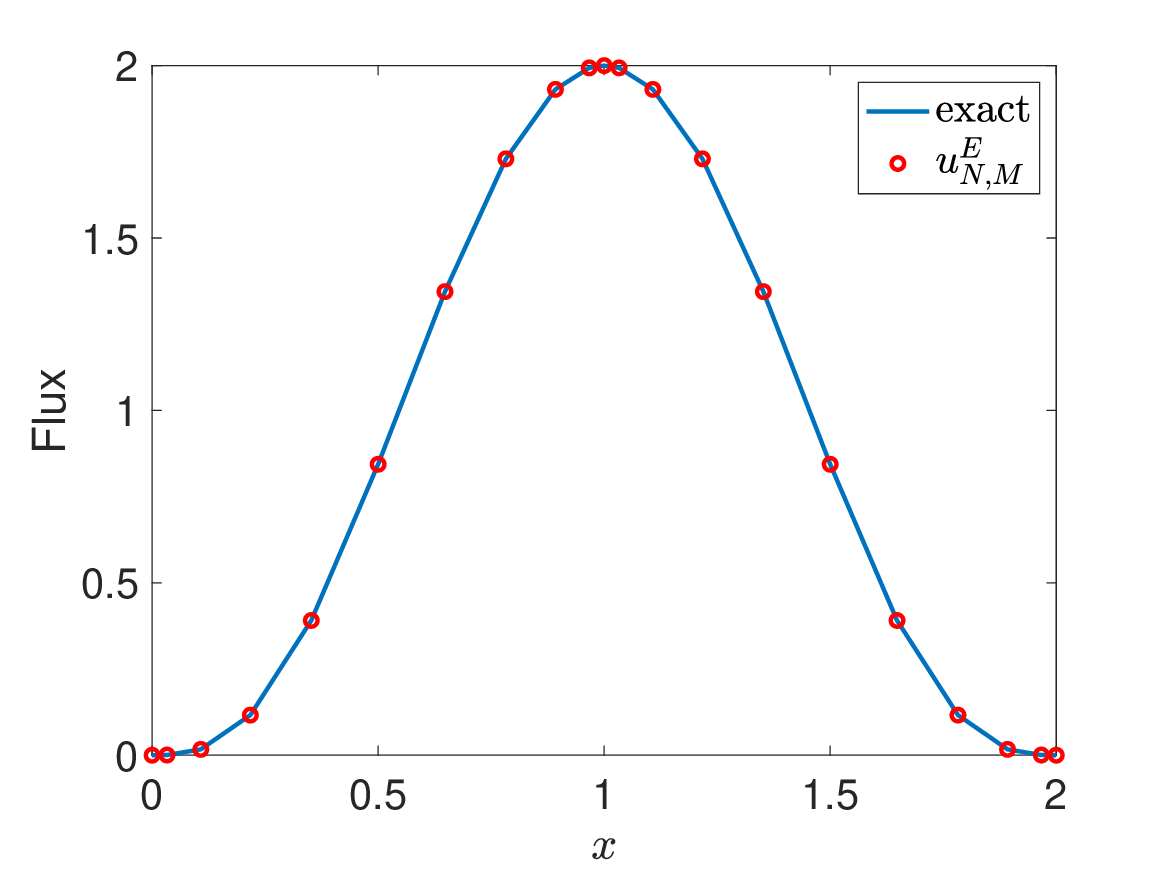}
				\label{ex1N10}
			\end{minipage}
		}
		\hspace{0.4em}
		\subfigure[Flux errors $\| u - u_{num} \|_{L^{2}}$ vs.  the element numbers $E$  by the quartic finite element method.]{
			\begin{minipage}[t]{0.45\linewidth}
				\centering
				\includegraphics[width=\linewidth]{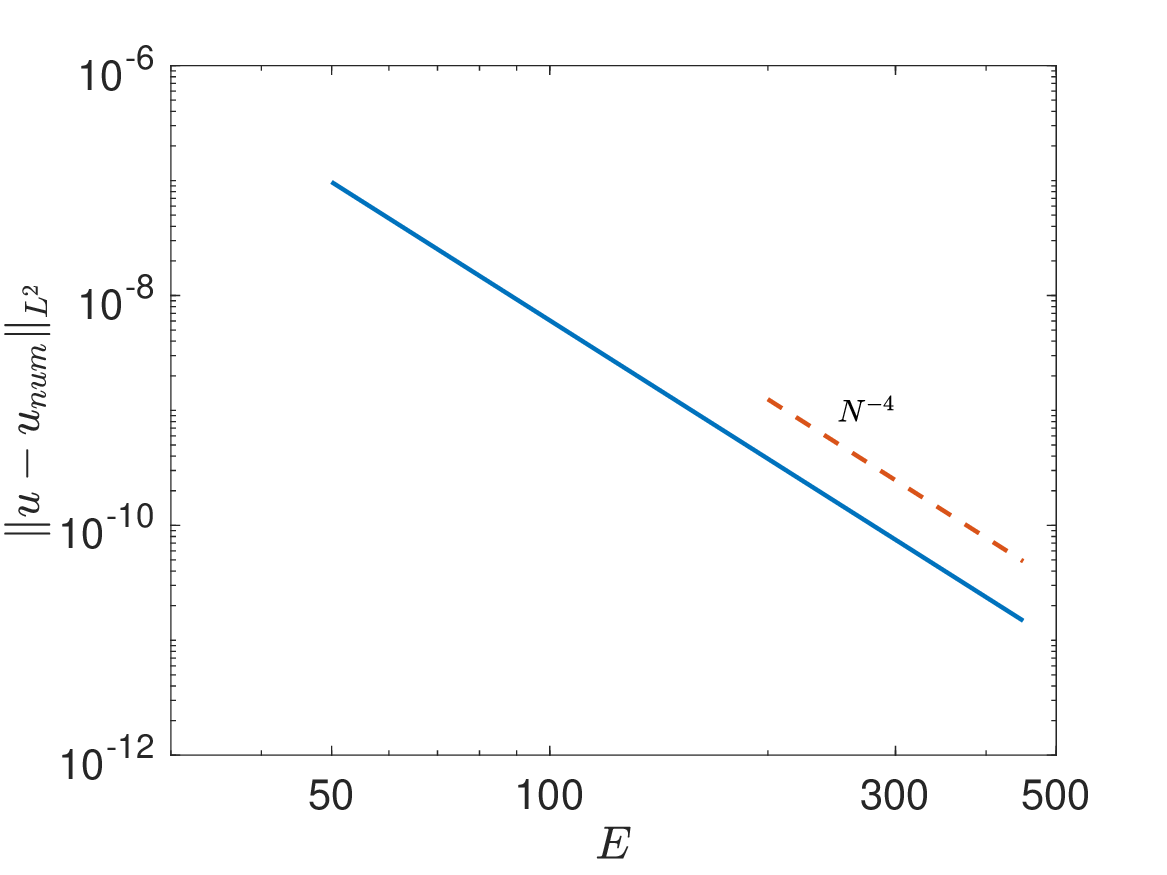}
				\label{errorl_inf_fem}
			\end{minipage}
		}
		\caption{The numerical flux by using the multi-domain spectral method (a), and $L^{2}$-errors of the numerical flux  vs.  $E$ by  the quartic finite element method (b).}
		\label{discontinex}
	\end{figure}
	Since the coefficients  are discontinuous at $x=1$, we partition the interval into two sub-intervals  $(0,1)$ and $(1,2)$ (i.e., $E=2$),   and apply the spectral method in each cell 
	 with the polynomial degrees $N$ and  $M$  for spatial and directional variables, respectively.  The trial functions are  continuous across the interface of two  cells.
	 In comparison, we also utilize the $C^0$-conforming   quartic finite element method  in spatial space (using piecewise polynomials of  the fixed degree $N=4$)  in $E$ elements. 
	 
	  Figure \ref{ex1N10} presents the numerical flux $u_{N,M}^E$ for $N=10$ and $M+1=2$ of our two-domain $(E = 2)$ spectral method. 	Figure \ref{errorL2_ex1_N} illustrates its  $p$-convergence of  the numerical flux  with respect to the spatial polynomial degree $N$ for $M=1$.  
	  While Figure \ref{errorl_inf_fem} shows the $h$-convergence of the numerical flux $u_{N,M}^E$ respect to the number of elements $E$ in  the quartic  finite element method 
	  on uniform meshes  with  $M=1$.
	\begin{figure}[H]
		\centering
		\subfigure[Numerical errors $\| u - u_{N,M}^E\|_{L^2}$ vs. $N$ $(M+1 = 2, E = 2)$.]{
			\begin{minipage}[t]{0.45\linewidth}
				\centering
				\includegraphics[width=\linewidth]{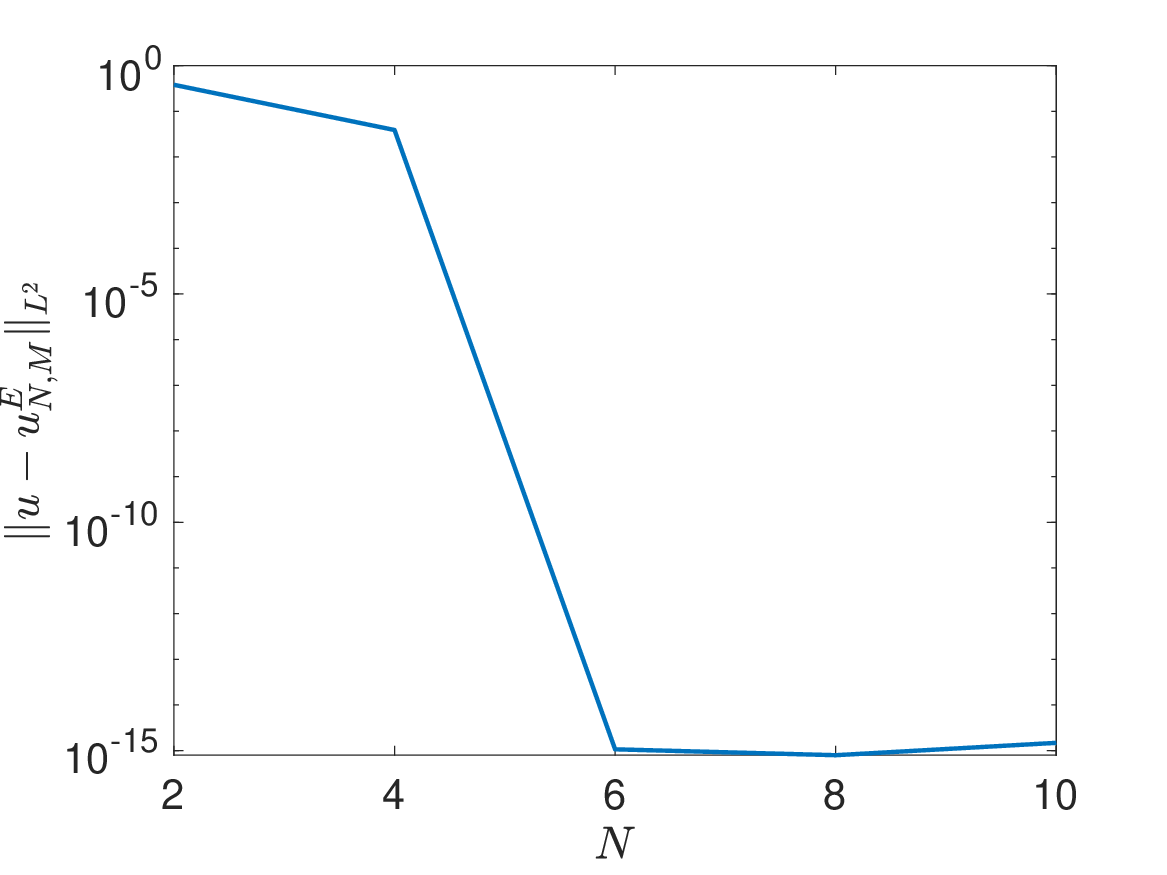}
				\label{errorL2_ex1_N}
			\end{minipage}
		}
		\hspace{0.4em}
		\subfigure[Numerical errors $\| u - u_{num}\|_{L^2}$ vs. $M$ evaluated by multi-domain spectral method and quartic FEM.]{
			\begin{minipage}[t]{0.45\linewidth}
				\centering
				\includegraphics[width=\linewidth]{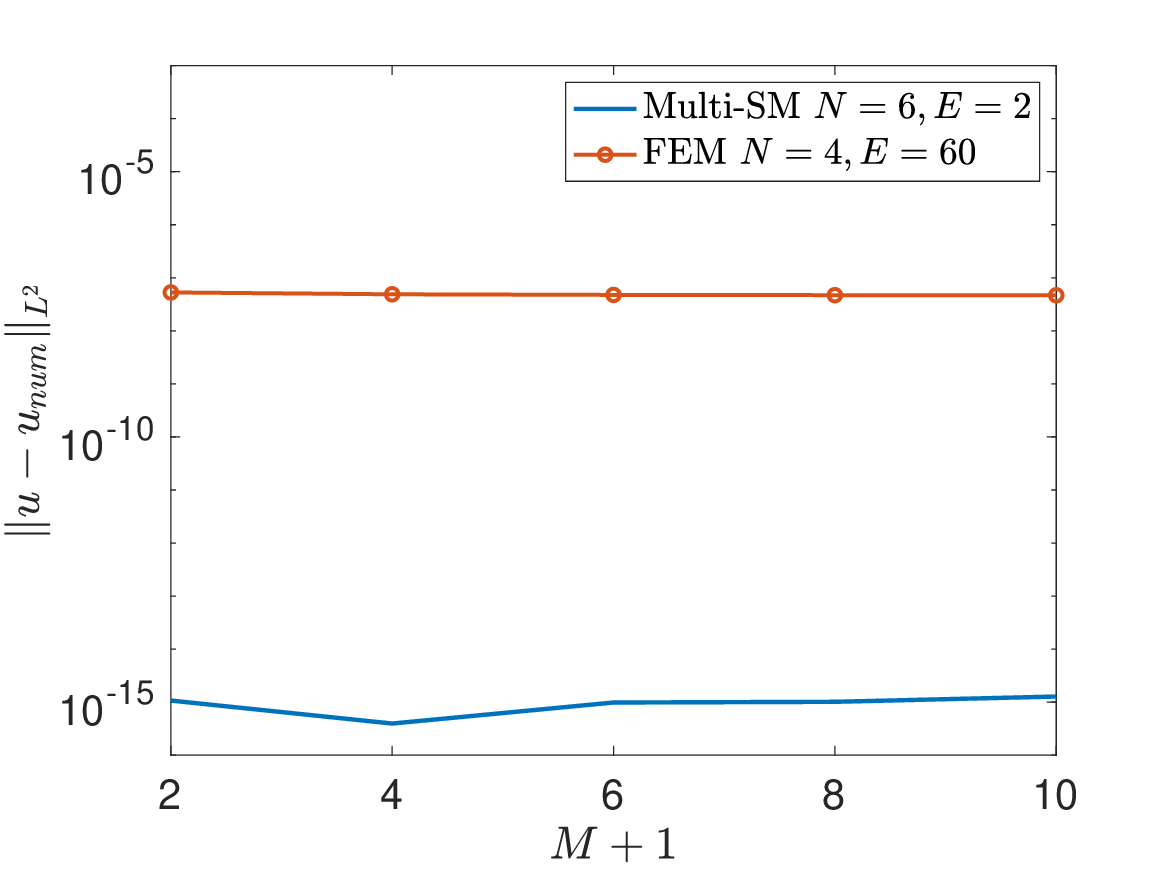}
				\label{errorL2_ex1_M}
			\end{minipage}
		}
		\caption{The $L^2$ numerical errors of flux versus $N$ for Example 5 evaluated by multi-domain spectral method and $L^2$-errors of the flux vs. $M$ by using the multi-domain spectral method and FEM.}
		\label{discontinexerror}
	\end{figure}

As one can observed, the multi-domain spectral still possesses an exponential order of  $p$-convergence in $N$ owing to the infinitely smooth solution.
While the finite element method can only enjoy a fixed order of $h$-convergence in $E$, more specifically, the flux errors measured in $L^{2}$-norm decay asymptotically  in $\mathcal{O}(h^{N})$ where $h=1/E$.  
	
Also, we plot the numerical errors versus the angular polynomial order $M$ in Figure \ref{errorL2_ex1_M} for both methods with other parameters fixed. Ideally,  no approximation errors will occur  in the angular discretization for both methods whenever $M\ge 1$ such that the total errors are only dependent of the spatial discretization parameters, since the source term, the coefficients and the exact solution are polynomials of degree $\le 1$ in the variable $\mu$. 
From  Figure \ref{errorL2_ex1_M}, one observes that  the multi-domain spectral method yields better results under the same degrees of freedom. 
		
	\subsection{Example 6}
	The specifications of the problem are defined by $D = (0,2)$,
	\[
	\Sigma_t = \left\{ 
	\begin{array}{ll}
		1,    & 0 \leq x \leq 1, \\[0.5em]	
		100,  & 1 < x \leq 2,	
	\end{array}
	\right. 
	\quad
	\Sigma_s = \left\{ 
	\begin{array}{ll}
		0,    & 0 \leq x \leq 1, \\[0.5em]	
		99.992,  & 1 < x \leq 2,	
	\end{array}
	\right. 
	\quad
	s = \left\{ 
	\begin{array}{ll}
		0,    & 0 \leq x \leq 1, \\[0.5em]	
		0.01,  & 1 < x \leq 2.	
	\end{array}
	\right.
	\]
	The $M+1=12$ Legendre-Gauss quadrature nodes are used for the angular discretization. The incoming angular flux at $x = 0$
	 changes linearly from 0 to 5 for the six discrete incoming directions, i.e,   $g_l(\mu_{i})= 11-i,\, i=6,7,\dots,11$. On the right boundary at $x = 2$, the vacuum boundary $g_r(\mu_{i})= 0, \, i=0,1,\dots,5$  is considered.
	
	Due to the discontinuity of the coefficients at $x = 1$, we segment the interval into $(0,1)$ and $(1,2)$, and implement our spectral method in each cell. For each spatial variable in each sub-region, we use an $N$th-degree polynomial approximation. Specifically, in each sub-region, we use Lagrange polynomials obtained from $N+1$ Gauss-Lobatto-Legendre points as the approximation function. For the directional variable, we use the Gauss-Legendre discrete ordinate method with $M+1=12$ for the approximation. This approach leads to our proposed spectral discretization scheme. Figure \ref{ex2flux} shows the flux at Gauss-Lobatto points of Example 6 obtained using a multi-domain spectral method with $N = 45$, $M+1 = 12$, and the reference flux is obtained from \cite{ren2022high}, which is evaluated by the HWENO method with  ${12}$ discrete ordinates  in the angular discretization. Figure \ref{fluxex2uniform} displays the flux at the uniform grid points of Example 6 obtained using multi-domain spectral method with $N = 30$ and $M+1 = 12$. Figure \ref{error_ex2_point} presents the pointwise flux errors at the uniform grid points, evaluated using the multi-domain spectral method with $N = 30$ and $M+1 = 12$.
	\begin{figure}[H]
		\centering
		\subfigure[Flux at Gauss-Lobatto points with $N=45$ and $M+1=12$.]{
			\begin{minipage}[t]{0.3\linewidth}
				\centering
				\includegraphics[width=\linewidth]{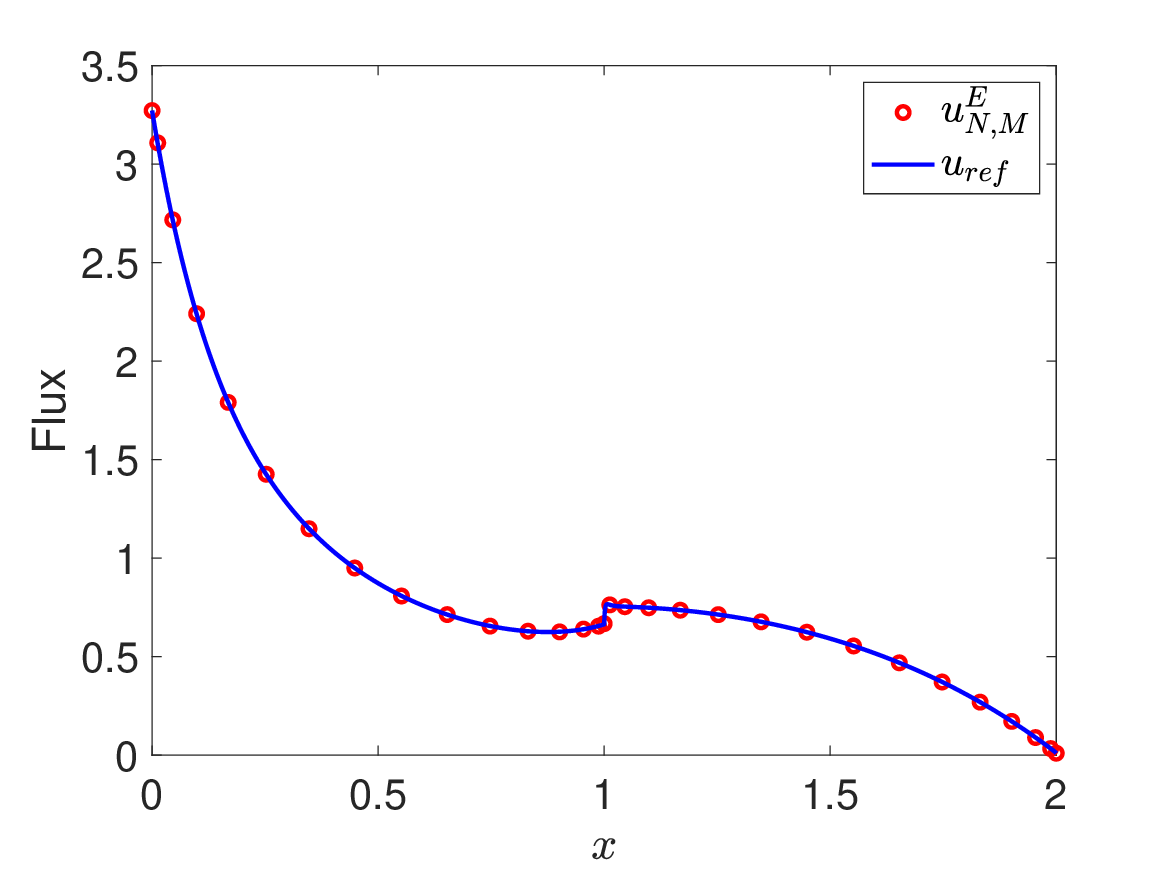}
				\label{ex2flux}
			\end{minipage}
		}
		\hspace{0.3em}
		\subfigure[Flux at uniform grid points with $N=30$ and $M+1=12$.]{
			\begin{minipage}[t]{0.3\linewidth}
				\centering
				\includegraphics[width=\linewidth]{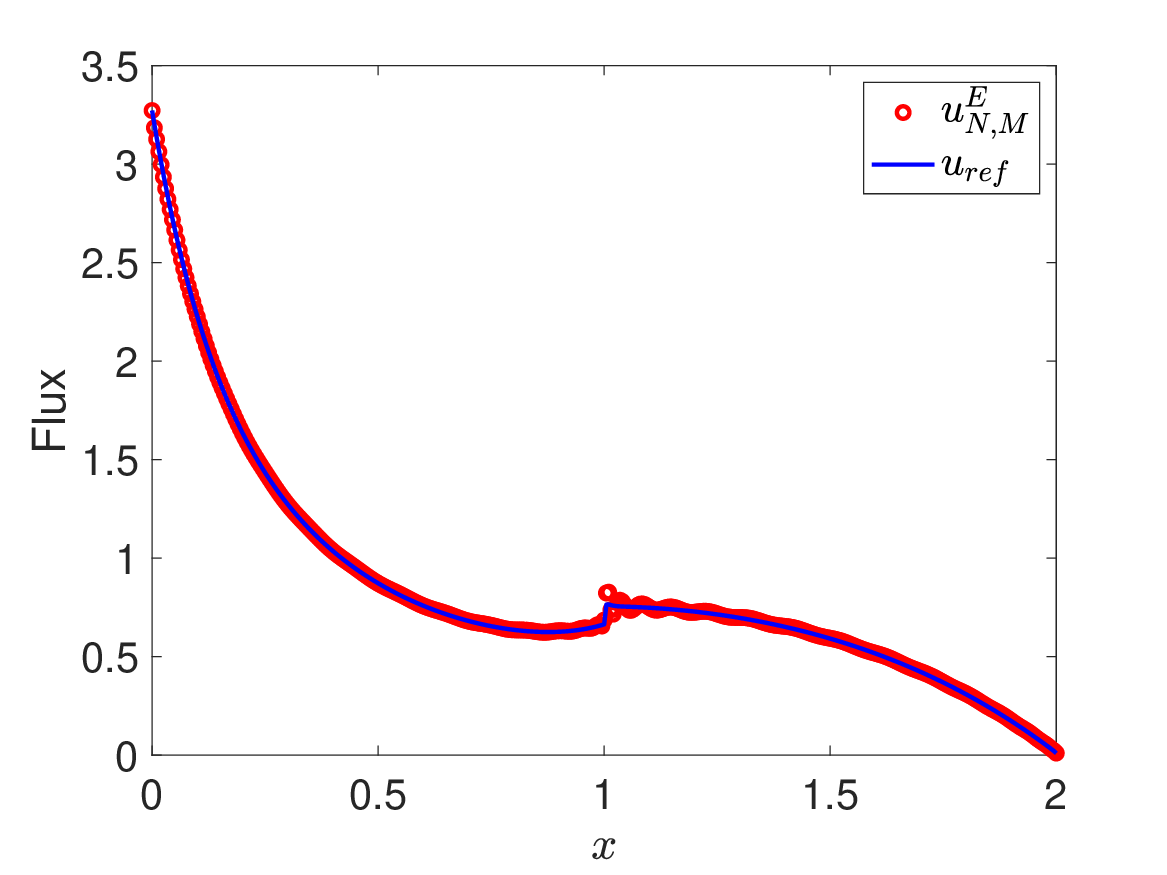}
				\label{fluxex2uniform}
			\end{minipage}
		}
		\hspace{0.3em}
		\subfigure[Pointwise flux errors at the uniform grid points.]{
			\begin{minipage}[t]{0.3\linewidth}
				\centering
				\includegraphics[width=\linewidth]{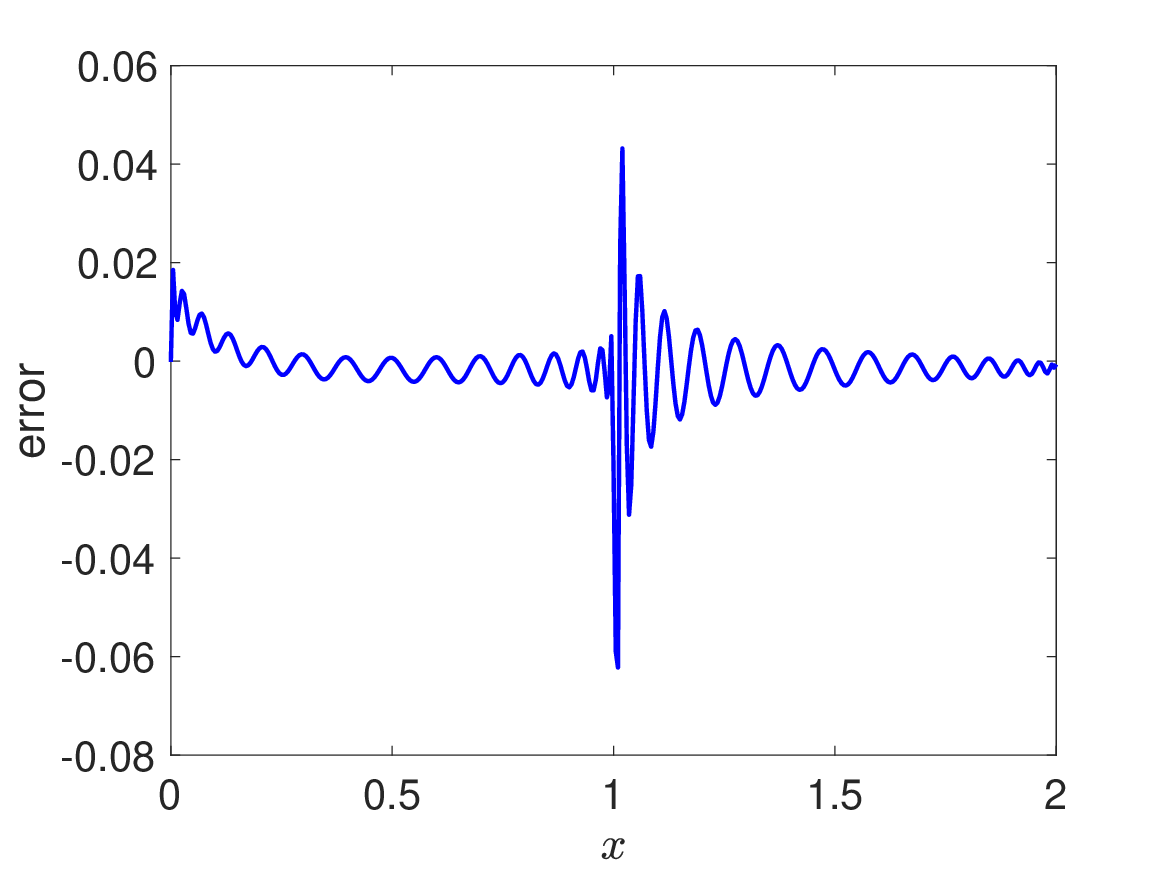}
				\label{error_ex2_point}
			\end{minipage}
		}
		\caption{Fluxes of the neutron transport equation with discontinuous coefficients evaluated by a multi-domain spectral method, and pointwise flux errors of the numerical solution.}
		\label{discontinex2}
	\end{figure}
	From Figure \ref{discontinex2}, it is evident that the numerical flux obtained using the multi-domain spectral method closely aligns with the reference solution flux. This indicates the effectiveness of the multi-domain spectral method in addressing the neutron transport problem with discontinuous coefficients.
	\begin{figure}[H]
		\centering
		\subfigure[Numerical error of flux $\|u - u_N^M \|_{L^{\infty}}$ vs. $N$ (spatial polynomial degree) obtained by multi-domain spectral method $(M = 11)$.]{
			\begin{minipage}[t]{0.45\linewidth}
				\centering
				\includegraphics[width=\linewidth]{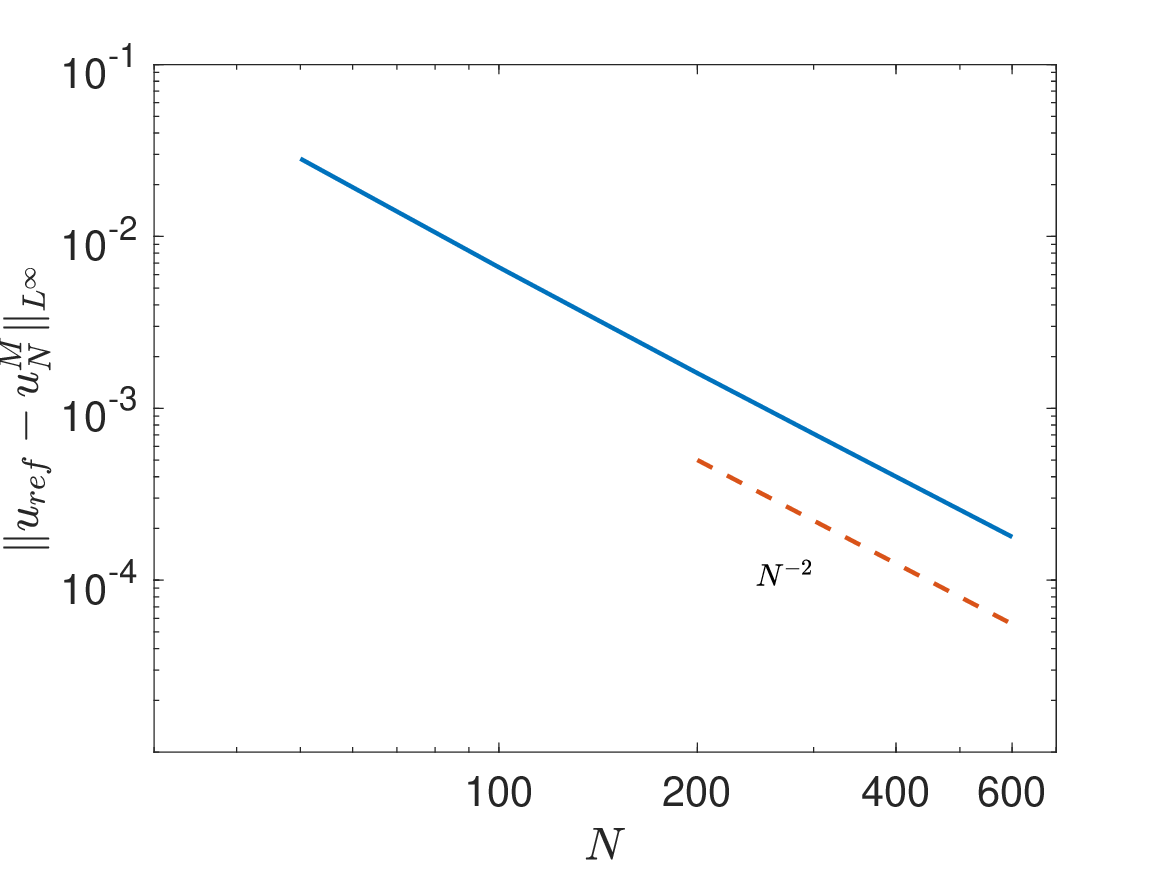}
				\label{errspmex2}
			\end{minipage}
		}
		\hspace{0.4em}
		\subfigure[The $L^{\infty}$-error of flux vs. $E$ (number of elements) by the quartic finite element method and  the HWENO method $(M = 11)$.]{
			\begin{minipage}[t]{0.45\linewidth}
				\centering
				\includegraphics[width=\linewidth]{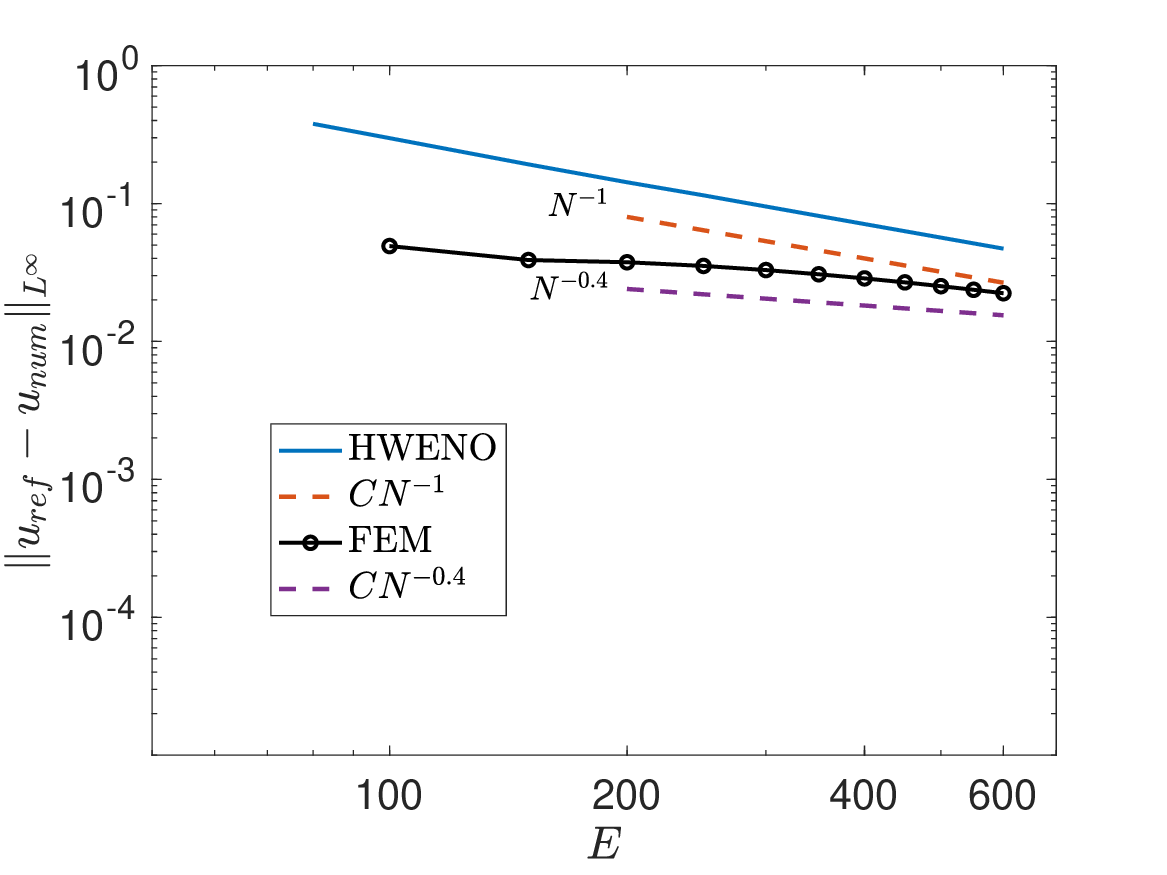}
				\label{errhwenoex2}
			\end{minipage}
		}
		\caption{The $L^{\infty}$ numerical errors of flux vs. $N$ evaluated by a multi-domain spectral method (a), the $L^{\infty}$ numerical errors of flux vs. $E$ by using HWENO method and quartic finite element method (b).}
		\label{errordisconex2}
	\end{figure}
	In Figure \ref{errordisconex2}, we present  the  $L^{\infty}$-errors of the numerical flux with respect to the parameter  $N$ or $E$  for  Example 6 involving discontinuous coefficients, solved using the multi-domain spectral method,   the HWENO method and the quartic FEM, all with $M+1=12$ discrete ordinates. From Example 1 in \cite{ren2022high}, we know that the HWENO method achieves an accuracy around $\mathcal{O}(h^{5.5})$ for problems with smooth solutions. For problems with discontinuous solutions, it  exhibits only a  first-order convergence (cf. Figure \ref{errhwenoex2}).  It can be seen clearly that our multi-domain spectral method in Example 6 achieves a higher convergence rate compared to the HWENO method and the quartic FEM.
	
	To verify the non-improvability of  the convergence order  in Theorem \ref{thm1}, we solve an example with an exact solution of finite regularity using spectral method. This enables us to evaluate the sharpness of  our error estimates  by examining the convergence order of our spectral method.
	\subsection{Example 7}\label{5.7}
	Here, we present an example with finite regularity to evaluate the sharpness of our error estimate. 
	
	The problem is characterized by the following specifications:
	$$D = (0,2), \quad \Sigma_t(x,\mu) = 3-x, \quad \Sigma_s(x,\mu) = 2-x, $$
	\[
	s(x,\mu) = \left\{ 
	\begin{array}{ll}
		2 \mu + 2x,    & 0 \leq x \leq 1, \\[0.5em]	
		-2 \mu +  2(2-x),  & 1 < x \leq 2.	
	\end{array}
	\right. 
	\]
	The exact solution 
	\[
	\varphi(x,\mu) = \left\{ 
	\begin{array}{ll}
		x,    & 0 \leq x \leq 1, \\[0.5em]	
		2-x,  & 1 < x \leq 2.	
	\end{array}
	\right. 
	\]
	We know that $|x| \in W
	_1^q(-1,1)$, for $ 2 < q < \infty$. From embedding theorem, we know $|x| \in W_2^s(-1,1)$, $s - 1/2 = 1 - 1/q$. Thus, $|x| \in H^{3/2-\epsilon}(-1,1)$ and $\varphi(\cdot,\mu) \in H^{3/2-\epsilon}(0,2)$, for any $ \epsilon > 0$. 
	\begin{figure}[H]
		\centering
		\subfigure[The flux with $N=60$ and $M+1=2$.]{
			\begin{minipage}[t]{0.45\linewidth}
				\centering
				\includegraphics[width=\linewidth]{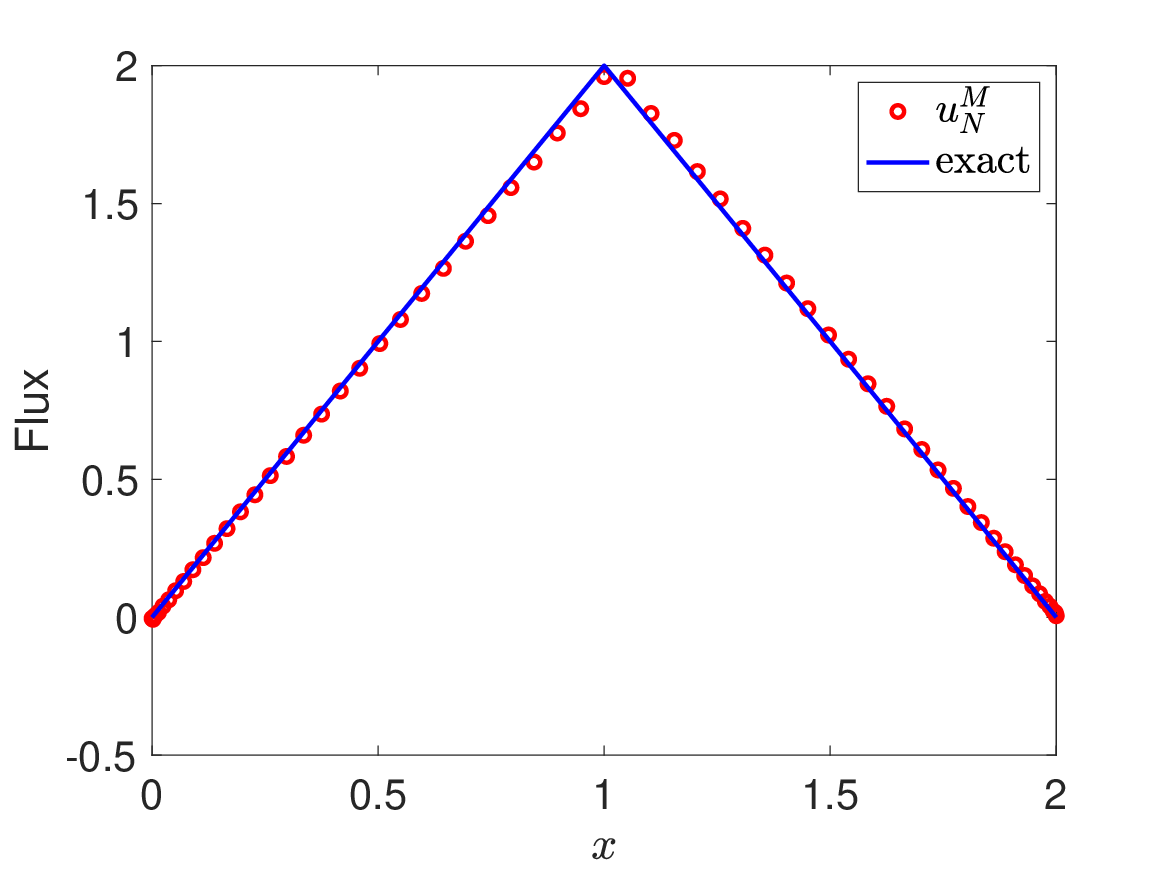}
				\label{ex2N60M2}
			\end{minipage}
		}
		\subfigure[Numerical errors {$\| u - u_N^M\|_{L^2(D)}+\| u - u_N^M\|_{L^2(\partial D)}$ vs. $N$ $(M+1 = 30)$}.]{
			\begin{minipage}[t]{0.45\linewidth}
				\centering
				\includegraphics[width=\linewidth]{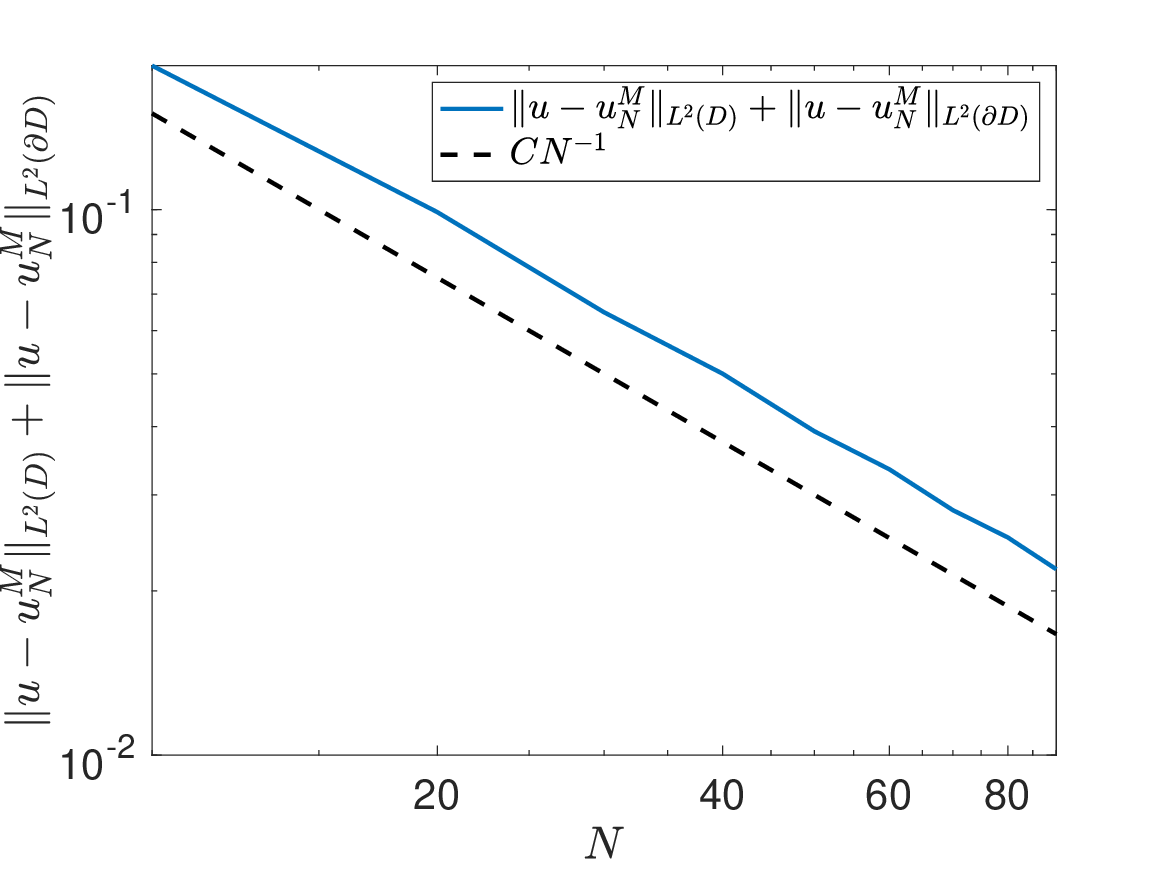}
				\label{errorL2_ex2_N}
			\end{minipage}
		}
		\caption{The flux of the solution with finite regularity evaluated by spectral method and $L^2$-errors of the flux of the numerical solution vs. $N$.}
		\label{extestconverg1}
	\end{figure}
	The flux and the $L^2$-errors of implementing the fully spectral discretization scheme for this problem are depicted in Figure \ref{extestconverg1}. Figure \ref{ex2N60M2} illustrates the numerical flux of the exact solution alongside the flux of the numerical solution, obtained using a spatial approximation polynomial of degree $N=60$ and $M+1=2$ discretization points in the direction of motion. Additionally, Figure \ref{errorL2_ex2_N} shows the $L^2$-errors of the numerical solution as a function of the spatial approximation polynomial degree $N$, with $M+1=30$.
	
	Based on the conclusions of our error analysis, the flux belongs to $H^{3/2-\epsilon}(D)$ for any $\epsilon > 0$, and the behavior of the $L^2$-errors with respect to the spatial approximation polynomial degree $N$ asymptotically obeys  $\mathcal{O}(N^{-1})$. From Figure \ref{errorL2_ex2_N}, we observe that the $L^2$-error of the numerical flux aligns with the conclusions of the error analysis. Thus, through this verification, we ascertain that the first term in Theorem \ref{thm1} is sharp.
%
	
	\section{Conclusion}\label{sec6}
	We have devised an efficient  (Petrov) Galerkin-Gauss collocation method to tackle the neutron transport equation. The solvability of this spectral approach has been substantiated, along with the provision of an error estimate for the transport equation. Furthermore, to exemplify the precision and efficacy of our approach, we present several  numerical examples and compare them with other methods. Through these comparisons, we observe that the fully spectral method offers advantages in both computational efficiency and error accuracy, in particular for those with a smooth solution.  
	
	It is worth mentioning that    the inflow boundary conditions are can also be naturally  embedded  in the variational form of our approximation scheme,
		\begin{align*}
			& \langle \mathcal{L}_M \varphi_N^M, v \rangle_{M}  + \langle | \mu | \varphi_{N}^M,  v  \rangle_{M,  \partial D^-} 
			=   \langle s, v\rangle_{M} + \langle   | \mu  | g, v \rangle_{M, \partial D^-}, \quad v\in W_N^M,
			\\
			& \langle u,  v  \rangle_{M,   \partial D^-}  = \sum_{0\le i\le M} \omega_i  u(-\sign(\mu_i), \mu_i) v(-\sign(\mu_i), \mu_i),
	\end{align*}
	which offers us a relative  easy  approach to extend our fully spectral method to high dimensions.
	Additionally, we emphasize our inclination to explore a spectral element approximation within the framework of the discontinuous Galerkin method for the neutron transport equation in our forthcoming work, aiming for flexibility on more general domains and adaptability to discontinuity solutions. 
		
	\section*{Acknowledgement} We would like to thank Dr. Yupeng Ren, the first author of \cite{ren2022high}, to provide the numerical data of HWENO method. We also show our appreciation to the anonymous referees for their comments and suggestions to improve the presentation.
	
	\vspace{1em}
	\noindent\textbf{Declaration of competing interest}
	\vspace{1em}
	
	The authors declare the following financial interests/personal relationships which may be considered as potential competing interests: 
	
	Huiyuan Li reports financial support was provided by National Key R\&D Program of China (No. 2021YFB0300203). Huiyuan Li reports financial support was provided by Huawei Technologies Co Ltd.
	
	\vspace{1em}
	\noindent\textbf{Data availability}
	\vspace{1em}
	
	Data will be made available on request.
	
	\bibliographystyle{abbrv}
	\bibliography{reference1d}

\begin{thebibliography}{10}

\bibitem{abbassi2011adaptive}
M.~Abbassi, A.~Zolfaghari, A.~Minuchehr, and M.~Yousefi.
\newblock An adaptive finite element approach for neutron transport equation.
\newblock {\em Nuclear Engineering and Design}, 241(6):2143--2154, 2011.

\bibitem{asadzadeh1998finite}
M.~Asadzadeh.
\newblock A finite element method for the neutron transport equation in an
  infinite cylindrical domain.
\newblock {\em SIAM Journal on Numerical Analysis}, 35(4):1299--1314, 1998.

\bibitem{canuto2007spectral}
C.~Canuto, M.~Y. Hussaini, A.~Quarteroni, and T.~A. Zang.
\newblock {\em Spectral Methods: Fundamentals in Single Domains}.
\newblock Springer Science \& Business Media, 2007.

\bibitem{cardona1994solution}
A.~Cardona and M.~Vilhena.
\newblock {Eine} {L{\"o}sung} der linearen {Transportgleichung mit Hilfe
  Tschebyscheffscher Polynome und der Laplace-Transformation}.
\newblock {\em Kerntechnik}, 59(6):278--281, 1994.

\bibitem{dautray1999mathematical}
R.~Dautray and J.-L. Lions.
\newblock {\em Mathematical Analysis and Numerical Methods for Science and
  Technology: Volume 6 Evolution Problems II}.
\newblock Springer Science \& Business Media, 2012.

\bibitem{ceedYD}
Y.~Dudouit.
\newblock Efficient high-dimension high-order matrix-free discontinuous
  {G}alerkin methods for high-fidelity physics.
\newblock In {\em CEED Meeting}, Lawrence Livermore National Laboratory, August
  2023.

\bibitem{fournier2013discontinuous}
D.~Fournier, R.~Herbin, and R.~L. Tellier.
\newblock Discontinuous {Galerkin} discretization and hp-refinement for the
  resolution of the neutron transport equation.
\newblock {\em SIAM Journal on Scientific Computing}, 35(2):A936--A956, 2013.

\bibitem{galliara1979finite}
J.~Galliara and M.~Williams.
\newblock A finite element method for neutron transport—{II}. some practical
  considerations.
\newblock {\em Annals of Nuclear Energy}, 6(4):205--223, 1979.

\bibitem{kadem2007solution}
A.~Kadem.
\newblock Solution of the three-dimensional transport equation using the
  spectral methods.
\newblock {\em Kybernetes}, 36(2):236--252, 2007.

\bibitem{lesaint1974finite}
P.~Lesaint and P.~A. Raviart.
\newblock On a finite element method for solving the neutron transport
  equation.
\newblock {\em Publications des s\'eminaires de math\'ematiques et informatique
  de Rennes}, (S4):1--40, 1974.

\bibitem{lewis1984computational}
E.~E. Lewis and W.~F. Miller.
\newblock {\em Computational Methods of Neutron Transport}.
\newblock John Wiley and Sons, Inc., New York, NY, 1984.

\bibitem{lewis1984monte}
E.~E. Lewis and W.~F. Miller.
\newblock Monte {Carlo} calculations of the transport of high-energy nucleons.
\newblock {\em La Rivista del Nuovo Cimento}, 7(9):1--48, 1984.

\bibitem{li2008iterative}
B.-W. Li, Y.-S. Sun, and Y.~Yu.
\newblock Iterative and direct {Chebyshev} collocation spectral methods for
  one-dimensional radiative heat transfer.
\newblock {\em International Journal of Heat and Mass Transfer},
  51(25-26):5887--5894, 2008.

\bibitem{machorro2007discontinuous}
E.~Machorro.
\newblock Discontinuous {Galerkin} finite element method applied to the {1-D}
  spherical neutron transport equation.
\newblock {\em Journal of Computational Physics}, 223(1):67--81, 2007.

\bibitem{martin1977finite}
W.~R. Martin and J.~J. Duderstadt.
\newblock Finite element solutions of the neutron transport equation with
  applications to strong heterogeneities.
\newblock {\em Nuclear Science and Engineering}, 62(3):371--390, 1977.

\bibitem{miller1973application}
W.~Miller~Jr, E.~Lewis, and E.~Rossow.
\newblock The application of phase-space finite elements to the one-dimensional
  neutron transport equation.
\newblock {\em Nuclear Science and Engineering}, 51(2):148--156, 1973.

\bibitem{MisunMinSEM}
M.~Min.
\newblock Spectral element methods for multiphysics transport problems.
\newblock In {\em International Conference on Spectral and High Order Methods},
  Yonsei University, Seoul, Korea, August 2023.

\bibitem{reed1973triangularmesh}
W.~H. Reed and T.~Hill.
\newblock Triangular mesh methods for the neutron transport equation.
\newblock {\em Los Alamos Report LA-UR-73-479}, 1973.

\bibitem{ren2022high}
Y.~Ren, Y.~Xing, D.~Wang, and J.~Qiu.
\newblock High order asymptotic preserving {H}ermite {WENO} fast sweeping
  method for the steady-state ${S_N}$ transport equations.
\newblock {\em Journal of Scientific Computing}, 93(1):3, 2022.

\bibitem{shen2011spectral}
J.~Shen, T.~Tang, and L.-L. Wang.
\newblock {\em Spectral Methods: Algorithms, Analysis and Applications},
  volume~41.
\newblock Springer Science \& Business Media, 2011.

\bibitem{verdu1994review}
G.~Verdu.
\newblock Review of the {Monte Carlo} method for particle transport.
\newblock {\em Journal of Nuclear Science and Technology}, 31(9):815--834,
  1994.

\bibitem{vilhena1999solutions}
M.~Vilhena, L.~B. Barichello, J.~Zabadal, C.~F. Segatto, A.~V. Cardona, and
  R.~Pazos.
\newblock Solutions to the multidimensional linear transport equation by the
  spectral method.
\newblock {\em Progress in Nuclear Energy}, 35(3-4):275--291, 1999.

\bibitem{wang2008simulation}
K.~Wang, P.-C. Long, H.-C. Wu, and M.-Y. Ye.
\newblock Simulation of {3D} neutron transport problems by a new method of
  finite difference.
\newblock {\em Annals of Nuclear Energy}, 35(3):488--493, 2008.

\bibitem{yamamoto2003solution}
A.~Yamamoto.
\newblock Solution of the {Boltzmann} transport equation for monoenergetic
  neutrons using the discrete ordinates and the finite difference methods.
\newblock {\em Progress in Nuclear Energy}, 43(1-4):41--92, 2003.

\bibitem{yuan2016high}
D.~Yuan, J.~Cheng, and C.-W. Shu.
\newblock High order positivity-preserving discontinuous {Galerkin} methods for
  radiative transfer equations.
\newblock {\em SIAM Journal on Scientific Computing}, 38(5):A2987--A3019, 2016.

\end{thebibliography}
\end{document}